\documentclass[a4paper,11pt]{amsart}
\usepackage{amssymb,dsfont,interval}
\usepackage[english]{babel}
\usepackage{hyperxmp}
\usepackage{cmap}
\usepackage{graphicx}
\usepackage{subfigure}
\usepackage{xcolor}
\usepackage[pdfdisplaydoctitle=true,
      colorlinks=true,
      urlcolor=blue,
      citecolor=blue,
      linkcolor=blue,
      pdfstartview=FitH,
      pdfpagemode=None,
      bookmarksnumbered=true]{hyperref} 
\usepackage[all]{xy}
\usepackage{mathrsfs}
\usepackage{booktabs}
\usepackage{enumerate}
\usepackage{float}
\graphicspath{{fig/}}
\newtheorem{thm}{Theorem}[section]
\newtheorem{cor}[thm]{Corollary}
\newtheorem{lem}[thm]{Lemma}

\theoremstyle{definition}
\newtheorem{defn}[thm]{Definition}
\theoremstyle{remark}
\newtheorem{rem}[thm]{Remark}
\newtheorem{ex}[thm]{Example}
\numberwithin{equation}{section}

\newcommand{\RR}{\mathbb{R}}        
\newcommand{\ZZ}{\mathbb{Z}}        

\newcommand{\Hn}[1]{\mathcal{H}_{#1}(\RR^{3})}      
\newcommand{\Pn}[1]{\mathcal{P}_{#1}(\RR^{3})}      

\newcommand{\SO}{\mathrm{SO}}       
\newcommand{\Pie}{\mathbb{P}\mathrm{iez}}    
\newcommand{\OO}{\mathrm{O}}        
\newcommand{\triv}{\mathds{1}}		
\newcommand{\DD}{\mathbb{D}}        
\newcommand{\octa}{\mathbb{O}}      
\newcommand{\ico}{\mathbb{I}}       
\newcommand{\tetra}{\mathbb{T}}     

\newcommand{\xx}{\mathbf{x}}   
\newcommand{\vv}{\mathbf{v}}   
\newcommand{\ww}{\mathbf{w}}   
\newcommand{\ii}{\mathbf{i}}   
\newcommand{\jj}{\mathbf{j}}   
\newcommand{\kk}{\mathbf{k}}   

\newcommand{\Idd}{I}

\newcommand{\avt}{\mathbf{vt}}
\newcommand{\aet}{\mathbf{et}}
\newcommand{\afc}{\mathbf{fc}}
\newcommand{\avc}{\mathbf{vc}}
\newcommand{\aec}{\mathbf{ec}}
\newcommand{\afd}{\mathbf{fd}}
\newcommand{\avd}{\mathbf{vd}}
\newcommand{\aed}{\mathbf{ed}}

\newcommand{\QQ}{\mathbf{Q}}   

\newcommand{\bp}{\mathrm{p}} 

\newcommand{\set}[1]{\left\{#1\right\}}  

\hypersetup{
  pdfauthor={M. Olive}, 
  pdftitle={Effective computation of $\SO(3)$ and $\OO(3)$ linear representations symmetry classes}, 
  pdfsubject={MSC 2010: 20C35, 22C05}, 
  pdfkeywords={Isotropy classes; Clips; Linear representations; Symmetry classes; Tensorial representations}, 
  pdflang=en, 
  }

\begin{document}

\title[Effective computation of isotropy classes]{Effective computation of $\SO(3)$ and $\OO(3)$ linear representations symmetry classes}%

\author{M. Olive}
\address{LMT (ENS Paris-Saclay, CNRS, Université Paris-Saclay), F-94235 Cachan
Cedex, France}
\email{marc.olive@math.cnrs.fr}


\date{\today}%
\begin{abstract}
  We propose a general algorithm to compute all the symmetry classes of any $\SO(3)$ or $\OO(3)$ linear representation. This method  relies on a binary operator between sets of conjugacy classes of closed subgroups, called the \emph{clips}. We compute explicit tables for this operation which allows to solve definitively the problem.
\end{abstract}

\maketitle


\begin{scriptsize}
  \setcounter{tocdepth}{2}
  \tableofcontents
\end{scriptsize}

\section{Introduction}

The problem of finding the \emph{symmetry classes} (also called \emph{isotropy classes}) of a given Lie group linear representation is a difficult task in general, even for a compact group, where their number is known to be finite~\cite{Mos1957,Man1962}.

It is only in 1996 that Forte--Vianello~\cite{FV1996} were able to define clearly the \emph{symmetry classes} of tensor spaces. Such tensor spaces, with natural $\OO(3)$ and $\SO(3)$ representations, appear in continuum mechanics \emph{via} linear constitutive laws. Thanks to this clarification, Forte--Vianello obtained for the first time the 8 symmetry classes of the $\SO(3)$ reducible representation on the space of \emph{Elasticity tensors}. One goal was to clarify and correct all the attempts already done to model the notion of \emph{symmetry} in continuum mechanics, as initiated by Curie~\cite{Cur1894}, but strongly influenced by crystallography. The results were contradictory - some authors announced nine different elasticity anisotropies~\cite{LOV2013} while others announced 10 of them~\cite{Sch1989,Jag1955,Fed2013,Gur1973}. Following Forte--Vianello, similar results were obtained in piezoelectricity~\cite{Wel2004}, photoelasticity~\cite{FV1997} and flexoelasticity~\cite{LQH2011}.

Besides these results on symmetry classes in continuum mechanics, the subject has also been active in the Mathematical community, especially due to its importance in Bifurcation theory. For instance, Michel in~\cite{Mic1980}, obtained the isotropy classes for \emph{irreducible} $\SO(3)$ representations. These results were confirmed by Ihrig and Golubitsky~\cite{IG1984} and completed by the symmetry classes for $\OO(3)$. Later, they were corrected by Chossat \& all in~\cite{CLRM1991}. Thereafter, Chossat--Guyard~\cite{CG1994} calculated the symmetry classes of a direct sum of \emph{two irreducible representations of $\SO(3)$}.

In this paper, we propose an algorithm -- already mentioned in~\cite{OA2013,OA2014,Oli2014} -- to obtain the finite set of symmetry classes for any $\OO(3)$ or $\SO(3)$ representation. Such an algorithm uses a binary operation defined over the set of conjugacy classes of a given group $G$ and that we decided to call the \emph{clips operation}. This operation was almost formulated in~\cite{CG1994}, but with no specific name, and only computed for $\SO(3)$ closed subgroups. As mentioned in~\cite{CG1996}, the clips operation allows to compute the set of \emph{symmetry classes} $\mathfrak{I}(V)$ of a direct sum $V=V_{1}\oplus V_{2}$ of linear representations of a group $G$, if we know the symmetry classes for each individual representations $\mathfrak{I}(V_{1})$ and $\mathfrak{I}(V_{2})$.

We compute clips tables for all conjugacy classes of closed subgroups of $\OO(3)$ and $\SO(3)$. The clips table for $\SO(3)$ was obtained in~\cite{CG1994}, but we give here more detail on the proof and extend the calculation to conjugacy classes of closed $\OO(3)$ subgroups. These results allow to obtain, in a finite step process, the set of symmetry classes for any \emph{reducible} $\OO(3)$ or $\SO(3)$-representation, so that for instance we directly obtain the 16 symmetry classes of \emph{Piezoelectric tensors}.

The subject being particularly important for applications to tensorial properties of any order, the necessity to convince of the correctness of the Clips tables require a complete/correct full article with sound proofs. Present article is therefore intended to be a final point to the theoretical problem and to the effective calculation of symmetry classes. Of course, we try to show the direct interest in the mechanics of continuous media (for this references to~\cite{OA2013,OA2014} are important), but we have no other choice than to insist on the mathematical formulation of the problem. 

The paper is organized as follow. In~\autoref{sec:clips-theory}, which is close to~\cite[Section 2.1]{CG1996}, the theory of \emph{clips} is introduced for a general group $G$ and applied in the context of symmetry classes where it is shown that isotropy classes of a direct sum corresponds to the clips of their respective isotropy classes. In~\autoref{sec:closed-subgroups}, we recall classical results on the classification of closed subgroups of $\SO(3)$ and $\OO(3)$ up to conjugacy. Models for irreducible representations of $\OO(3)$ and $\SO(3)$ and their symmetry classes are recalled in~\autoref{sec:irreducible-representations}. We then provide in~\autoref{subsec:Mechanical_Prop} some applications to tensorial mechanical properties, as the non classical example of Cosserat elasticity. The clips tables for $\SO(3)$ and $\OO(3)$ are presented in~\autoref{sec:clips-tables}. The details and proofs of how to obtain these tables are provided in~\autoref{sec:proofs-SO3} and~\autoref{sec:proofs-O3}.

\section{A general theory of clips}
\label{sec:clips-theory}

Given a group $G$ and a subgroup $H$ of $G$, the conjugacy class of $H$
\begin{equation*}
  [H] := \set{gHg^{-1},\quad g\in G}
\end{equation*}
is a subset of $\mathcal{P}(G)$. We define $\mathrm{Conj}(G)$ to be the set of all conjugacy classes of a given group $G$:
\begin{equation*}
  \mathrm{Conj}(G) := \set{[H],\quad H\subset G}.
\end{equation*}
Recall that, on $\mathrm{Conj}(G)$, there is a pre-order relation induced by inclusion. It is defined as follows:
\begin{equation*}
  [H_{1}] \preceq [H_{2}] \quad \text{if $H_{1}$ is conjugate to a subgroup of $H_{2}$}.
\end{equation*}
When restricted to the \emph{closed subgroups} of a topological compact group, this pre-order relation becomes a \emph{partial order}~\cite[Proposition 1.9]{Bre1972} and defines the \emph{poset} (partial ordered set) of conjugacy classes of \emph{closed subgroups} of $G$.

We now define a binary operation called the \emph{clips operation} on the set $\mathrm{Conj}(G)$.

\begin{defn}
  Given two conjugacy classes $[H_{1}]$ and $[H_{2}]$ of a group $G$, we define their \emph{clips} as the following subset of conjugacy classes:
  \begin{equation*}
    [H_{1}] \circledcirc [H_{2}] := \set{[H_{1} \cap gH_{2}g^{-1}],\quad g \in G}.
  \end{equation*}
  This definition immediately extends to two families (finite or infinite) $\mathcal{F}_{1}$ and $\mathcal{F}_{2}$ of conjugacy classes:
  \begin{equation*}
    \mathcal{F}_{1} \circledcirc \mathcal{F}_{2} := \bigcup_{[H_{i}] \in \mathcal{F}_{i}} [H_{1}] \circledcirc [H_{2}].
  \end{equation*}
\end{defn}

\begin{rem}
The clips operation was already introduced, with no specific name, in~\cite{CG1994}, the notation being $\mathbf{P}(H_1,H_2)$. In this article, the author only focus on the $\SO(3)$ case, with no meaning to deal with a general theory.
\end{rem}

This clips operation defined thus a binary operation on the set $\mathcal{P}(\mathrm{Conj}(G))$ which is \emph{associative} and \emph{commutative}. We have moreover
\begin{equation*}
  [\triv]\circledcirc [H] = \set{[\triv] } \text{ and } [G]\circledcirc [H] = \set{[H]},
\end{equation*}
for every conjugacy class $[H]$, where $\triv:=\set{e}$ and $e$ is the identity element of $G$.

Consider now a linear representation $(V,\rho)$ of the group $G$. Given $\vv\in V$, its \emph{isotropy group} (or \emph{symmetry group}) is defined as
\begin{equation*}
  G_{\vv} := \set{g\in G,\quad g\cdot\vv=\vv}
\end{equation*}
and its \emph{isotropy class} is the conjugacy class $[G_{\vv}]$ of its isotropy group. The \emph{isotropy classes} (or \emph{orbit types}) of the representation $V$ is the family of all isotropy classes of vectors $\vv$ in $V$:
\begin{equation*}
  \mathfrak{I}(V) := \set{[G_{\vv}]; \; \vv \in V}.
\end{equation*}
The central observation is that the isotropy classes of a direct sum of representations is obtained by the clips of their respective isotropy classes.

\begin{lem}\label{lem:Clips_Direct_Sum}
  Let $V_{1}$ and $V_{2}$ be two linear representations of $G$. Then
  \begin{equation*}
    \mathfrak{I}(V_{1}\oplus V_{2})=\mathfrak{I}(V_{1})\circledcirc \mathfrak{I}(V_{2}).
  \end{equation*}
\end{lem}

\begin{proof}
  Let $[G_{\vv}]$ be some isotropy class for $\mathfrak{I}(V_{1}\oplus V_{2})$ and write $\vv=\vv_{1}+\vv_{2}$ where $\vv_{i}\in V_{i}$. Note first that $G_{\vv_{1}}\cap G_{\vv_{2}}\subset G_{\vv}$. Conversely given $g\in G_{\vv}$, we get
  \begin{equation*}
    g\cdot \vv=g\cdot \vv_{1}+g\cdot \vv_{2}=\vv,\quad g\cdot \vv_{i}\in V_{i},
  \end{equation*}
  and thus $g\cdot \vv_{i}= \vv_{i}$. This shows that $G_{\vv} = G_{\vv_{1}}\cap G_{\vv_{2}}$ and therefore that
  \begin{equation*}
    \mathfrak{I}(V_{1}\oplus V_{2})\subset \mathfrak{I}(V_{1})\circledcirc \mathfrak{I}(V_{2}).
  \end{equation*}
  Conversely, let $[H]=[H_{1}\cap gH_{2}g^{-1}]$ in $\mathfrak{I}(V_{1})\circledcirc \mathfrak{I}(V_{2})$ where $H_{i} = G_{\vv_{i}}$ for some vectors $\vv_{i} \in V_{i}$. Then, if we set
  \begin{equation*}
    \vv=\vv_{1}+g\cdot \vv_{2},
  \end{equation*}
  we have $G_{\vv}=H_{1}\cap gH_{2}g^{-1}$, as before, which shows that
  \begin{equation*}
    [H_{1}\cap gH_{2}g^{-1}] \in \mathfrak{I}(V_{1}\oplus V_{2})
  \end{equation*}
  and achieves the proof.
\end{proof}

Using this lemma, we deduce a general algorithm to obtain the isotropy classes $\mathfrak{I}(V)$ of a finite dimensional representation of a reductive algebraic group $G$, provided we know:
\begin{enumerate}
  \item a decomposition $V=\bigoplus_{i} W_{i}$ into irreducible representations $W_{i}$.
  \item the isotropy classes $\mathfrak{I}(W_{i})$ for the \emph{irreducible representations} $W_{i}$;
  \item the tables of \emph{clips operations} $[H_{1}]\circledcirc [H_{2}]$ between conjugacy classes of closed subgroups $[H_{i}]$ of $G$.
\end{enumerate}

In the sequel of this paper, we will apply successfully this program to the linear representations of $\SO(3)$ and $\OO(3)$.

\section{Closed subgroups of $\OO(3)$}
\label{sec:closed-subgroups}

Every closed subgroup of $\SO(3)$ is conjugate to one of the following list~\cite{GSS1988}
\begin{equation*}
  \SO(3),\, \OO(2),\, \SO(2),\, \DD_{n} (n \ge 2),\, \ZZ_{n} (n \ge 2),\, \tetra,\, \octa,\, \ico,\, \text{and}\, \triv
\end{equation*}
where:
\begin{itemize}
  \item $\OO(2)$ is the subgroup generated by all the rotations around the $z$-axis and the order 2 rotation $r : (x,y,z)\mapsto (x,-y,-z)$ around the $x$-axis;
  \item $\SO(2)$ is the subgroup of all the rotations around the $z$-axis;
  \item $\ZZ_{n}$ is the unique cyclic subgroup of order $n$ of $\SO(2)$ ($\ZZ_{1} = \set{\Idd}$);
  \item $\DD_{n}$ is the \emph{dihedral} group. It is generated by $\ZZ_{n}$ and $r :(x,y,z)\mapsto (x,-y,-z)$ ($\DD_{1} = \set{\Idd}$);
  \item $\tetra$ is the \emph{tetrahedral} group, the (orientation-preserving) symmetry group of the tetrahedron $\mathcal{T}_{0}$ defined in~\autoref{fig:cube0}. It has order 12;
  \item $\octa$ is the \emph{octahedral} group, the (orientation-preserving) symmetry group of the cube $\mathcal{C}_{0}$ defined in~\autoref{fig:cube0}. It has order 24;
  \item $\ico$ is the \emph{icosahedral} group, the (orientation-preserving) symmetry group of the dodecahedron $\mathcal{D}_{0}$ in~\autoref{fig:dode}. It has order 60;
  \item $\triv$ is the trivial subgroup, containing only the unit element.
\end{itemize}

The \emph{poset} of conjugacy classes of closed subgroups of $\SO(3)$ is completely described by the following inclusion of subgroups~\cite{GSS1988}:
\begin{align*}
    & \ZZ_{n} \subset \DD_{n} \subset \OO(2) \qquad (n \ge 2);                                          \\
    & \ZZ_{n} \subset \ZZ_{m} \text{ and } \DD_{n} \subset \DD_{m}, \qquad (\text{if $n$ divides $m$}); \\
    & \ZZ_{n} \subset \SO(2) \subset \OO(2) \qquad (n \ge 2);
\end{align*}
completed by $[\ZZ_{2}]\preceq [\DD_{n}]$ ($n\geq 2$) and by the arrows in~\autoref{fig:SO3-lattice} (note that an arrow between the classes $[H_{1}]$ and $[H_{2}]$ means that $[H_{1}]\preceq [H_{2}]$), taking account of the exceptional subgroups $\octa, \tetra, \ico$ .

\begin{figure}[h!]
  \centering
  \includegraphics[scale=0.8]{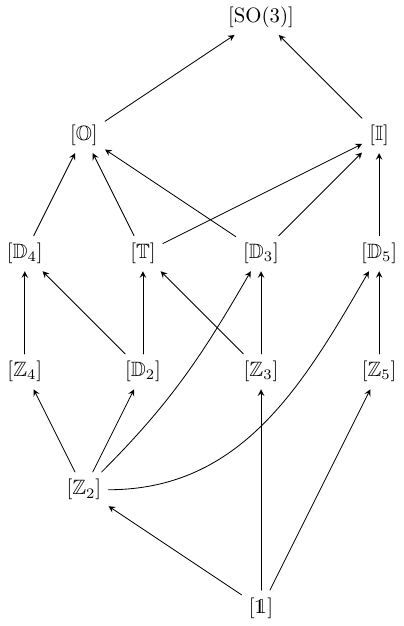}
  \caption{Exceptional conjugacy classes of closed $\SO(3)$ subgroups}
  \label{fig:SO3-lattice}
\end{figure}

Classification of $\OO(3)$-closed subgroups is more involving~\cite{IG1984,Ste1994} and has been described using \emph{three types of subgroups}. Given a closed subgroup $\Gamma$ of $\OO(3)$ this classification runs as follows.
\begin{description}
  \item[Type I] A subgroup $\Gamma$ is of type I if it is a subgroup of $\SO(3)$;
  \item[Type II] A subgroup $\Gamma$ is of type II if $-\Idd\in \Gamma$. In that case, $\Gamma$ is generated by some subgroup $K$ of $\SO(3)$ and $-\Idd$	;
  \item[Type III] A subgroup $\Gamma$ is of type III if $-\Idd \notin \Gamma$ and $\Gamma$ is not a subgroup of $\SO(3)$.
\end{description}

The description of type III subgroups requires more details. We will denote by $\QQ(\vv;\theta)\in \SO(3)$ the rotation around $\vv\in \RR^3$ with angle $\theta \in \interval[open right]{0}{2\pi}$ and by $\sigma_\vv\in \OO(3)$, the reflection through the plane normal to $\vv$. Finally, we fix an arbitrary orthonormal frame $(\ii,\jj,\kk)$, and we introduce the following definitions.

\begin{itemize}
  \item $\ZZ^{-}_{2}$ is the order $2$ reflection group generated by $\sigma_\ii$ (where $\ZZ_{1}^- = \set{\triv}$);
  \item $\ZZ^{-}_{2n}$ ($n\geq 2$) is the group of order $2n$, generated by $\ZZ_{n}$ and $-\QQ\left(\kk;\displaystyle{\frac{\pi}{n}}\right)$ (see~\eqref{eq:Z2nMoins});
  \item $\DD_{2n}^h$ ($n\geq 2$) is the group of order $4n$ generated by $\DD_{n}$ and $-\QQ\left(\kk;\displaystyle{\frac{\pi}{n}}\right)$ (see~\eqref{eq:D2nh});
  \item $\DD_{n}^{v}$ ($n\geq 2$) is the group of order $2n$ generated by $\ZZ_{n}$ and $\sigma_\ii$ (where $\DD_{1}^{v} = \set{\triv}$);
  \item $\OO(2)^{-}$ is generated by $\SO(2)$ and $\sigma_\ii$.
\end{itemize}

These planar subgroups are completed by the subgroup $\octa^{-}$ which is of order $24$ (see \autoref{subsubsec:OMoins} and~\eqref{eq:omoins} for details).

The \emph{poset} of conjugacy classes of closed subgroups of $\OO(3)$ is given in~\autoref{fig:cubic-sub-lattice} for $\octa^{-}$ subgroups and in~\autoref{fig:O3-lattice} for $\OO(3)$ subgroups.

\begin{figure}[h]
  \centering
  \includegraphics[scale=0.8]{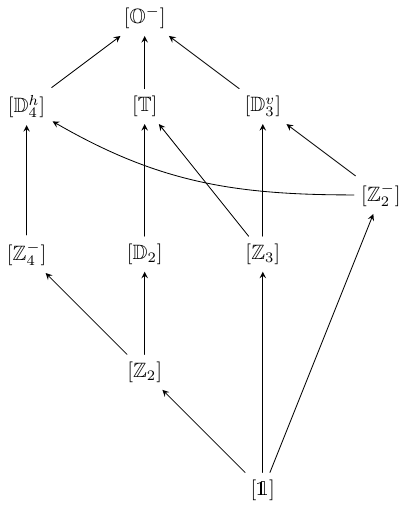}
  \caption{Poset of closed $\octa^{-}$ subgroups}
  \label{fig:cubic-sub-lattice}
\end{figure}

\begin{figure}[h]
  \centering
  \includegraphics[scale=0.8]{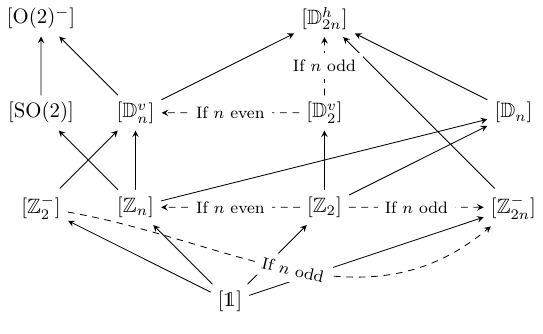}
  \caption{Poset of closed $\OO(3)$ subgroups}
  \label{fig:O3-lattice}
\end{figure}

\section{Symmetry classes for irreducible representations}
\label{sec:irreducible-representations}

Let $\Pn{n}$ be the space of \emph{homogeneous polynomials} of degree $n$ on $\RR^{3}$. We have two natural representations of $\OO(3)$ on $\Pn{n}$. The first one, noted $\rho_{n}$ is given by
\begin{equation*}
  [\rho_{n}(\bp)](\xx):=\bp(g^{-1}\xx),\quad g\in G,\quad \xx\in \RR^{3}
\end{equation*}
whereas the second one, noted $\rho^{*}_{n}$, is given by
\begin{equation*}
  [\rho^{*}_{n}(\bp)](\xx):=\det(g)\bp(g^{-1}\xx),\quad g\in G,\quad \xx\in \RR^{3}.
\end{equation*}
Note that both of them induce the same representation $\rho_{n}$ of $\SO(3)$.

Let $\Hn{n}\subset \Pn{n}$ be the subspace of \emph{homogeneous harmonic polynomials} of degree $n$ (polynomials with null Laplacian). It is a classical fact~\cite{GSS1988} that $(\Hn{n},\rho_{n})$ and $(\Hn{n},\rho_{n}^{*})$ ($n \ge 0$) are irreducible $\OO(3)$-representations, and each irreducible $\OO(3)$-representation is isomorphic to one of them. Models for irreducible representations of $\SO(3)$ reduce to $(\Hn{n},\rho_{n})$ ($n \ge 0$).

\begin{rem}
  Other classical models for $\OO(3)$ and $\SO(3)$ irreducible representations, used in mechanics~\cite{FV1996}, are given by spaces of \emph{harmonic tensors} (\textit{i.e.} totally symmetric traceless tensors).
\end{rem}

The isotropy classes for irreducible representations of $\SO(3)$ was first obtained by Michel~\cite{Mic1980}. Same results were then obtained and completed in the $\OO(3)$ case by Ihrig-Golubistky~\cite{IG1984} and then by Chossat and al~\cite{CLRM1991}.

\begin{thm}\label{thm:SO3_Irr_Isot}
  The isotropy classes for the $\SO(3)$-representation $(\Hn{n},\rho_{n})$ are:
  \begin{enumerate}
    \item $[\triv]$ for $n\geq 3$;
    \item $[\ZZ_k]$ for $2 \le k \le n$ if $n$ is odd and $2 \le k \le n/2$ if $n$ is even;
    \item $[\DD_k]$ for $2 \le k \le n$;
    \item $[\tetra]$ for $n=3,6$, $7$ or $n\geq 9$;\label{list:SO3IrredTetra}
    \item $[\octa]$ for $n\neq 1,2,3,5,7,11$;\label{list:SO3IrredOcta}
    \item $[\ico]$ for $n=6,10,12,15,16,18$ or $n\geq 20$ and $n\neq 23,29$;
    \item $[\SO(2)]$ for $n$ odd;
    \item $[\OO(2)]$ for $n$ even;
    \item $[\SO(3)]$ for any $n$.
  \end{enumerate}
\end{thm}

\begin{rem}
  The list in Theorem~\ref{thm:SO3_Irr_Isot} is similar to the list in~\cite[Table A1]{Mic1980} and~\cite[Table A4]{CLRM1991}. In~\cite[Theorem 6.6]{IG1984} (for $\SO(3)$ irreducible representations) :
  \begin{itemize}
    \item $[\tetra]$ is an isotropy class for $n=6,7$ and $n\geq 9$;
    \item $[\octa]$ is an isotropy class for $n\neq 1,2,5,7,11$.
  \end{itemize}
  Such lists are different from \eqref{list:SO3IrredTetra} and \eqref{list:SO3IrredOcta} in our Theorem~\ref{thm:SO3_Irr_Isot}. But according to~\cite[Proposition 3.7]{IG1984}, $[\tetra]$ is a maximum isotropy class for $n=3$. We have thus corrected this error in Theorem~\ref{thm:SO3_Irr_Isot}.
\end{rem}

\begin{thm}\label{thm:O3_Irr_Isot}
  The isotropy classes for the $\OO(3)$-representations $(\Hn{n},\rho_{n})$ (for $n$ odd) and $(\Hn{n},\rho_{n}^{*})$ (for $n$ even) are:
  \begin{enumerate}
    \item $[\triv]$ for $n\geq 3$;
    \item $[\ZZ_k]$ for $2\leq k\leq n/2$;
    \item $[\ZZ_{2k}^{-}]$ for $k\leq \dfrac{n}{3}$;
    \item $[\DD_k]$ for $2 \le k \le n/2$ if $n$ is odd and for  $2 \le k \le n$ if $n$ is even.
    \item $[\DD_k^{v}]$ for $2 \le k \le n$ if $n$ is odd and $2 \le k \le n/2$ if $n$ is even;
    \item $[\DD_{2k}^h]$ for $2 \le k\leq n$, except $\DD_4^h$ for $n=3$;
    \item $[\tetra]$ for $n\neq 1,2,3,5,7,8,11$;\label{list:O3IrredTetra}
    \item $[\octa]$ for $n\neq 1,2,3,5,7,11$;
    \item $[\octa^{-}]$ for $n\neq 1,2,4,5,8$;
    \item $[\ico]$ for $n=6,10,12,15,16,18$ or $n\geq 20$ and $n\neq 23,29$;
    \item $[\OO(2)]$ when $n$ is even ;
    \item $[\OO(2)^{-}]$ when $n$ is odd.
  \end{enumerate}
\end{thm}

\begin{rem}
  The list in Theorem~\ref{thm:O3_Irr_Isot} for $n$ odd is similar to the list in~\cite[Table A4]{CLRM1991}.	
  In~\cite[Theorem 6.8]{IG1984} (for $\OO(3)$ irreducible representations), $[\tetra]$ is an isotropy class for $n\neq 1,2,5,7,8,11$ (which is different from the list \eqref{list:O3IrredTetra} in our Theorem~\ref{thm:O3_Irr_Isot} above). But according to~\cite{GW2002,Wel2004} and~\cite[Table A4]{CLRM1991}, $[\tetra]$ is not an isotropy class in the case $n=3$, and we corrected this error in the list~\eqref{list:O3IrredTetra} of Theorem~\ref{thm:O3_Irr_Isot}.
\end{rem}

\section{Clips tables}
\label{sec:clips-tables}

\subsection{$\SO(3)$ closed subgroups}

The resulting conjugacy classes for the clips operation of closed $\SO(3)$ subgroups are given in~\autoref{tab:SO3-clips}.

The following notations have been used:
\begin{align*}
  d & := \gcd(m,n), & d_{2} & := \gcd(n,2),          & k_{2}  & := 3-d_{2},\\
  d_{3} & := \gcd(n,3), & d_{5} & := \gcd(n,5) \\
  dz    & :=2, \text{ if $m$ and $n$ even}, & dz     & :=1, \text{ otherwise},                \\
  d_4 & := 4, \text{ if } 4 \text{ divide } n, & d_4 & := 1, \text{ otherwise}, \\
  \ZZ_{1}&=\DD_{1}:=\triv.
\end{align*}
 
\begin{small}
  \begin{table}[H]
    \renewcommand{\arraystretch}{1.5}
    \renewcommand\tabcolsep{2pt}
    \begin{tabular}{|c|c|c|c|c|c|c|c|}
      \hline
      $\circledcirc$         & $\left[\ZZ_{n}\right]$         & $\left[\DD_{n}\right]$ & $\left[\tetra\right]$ & $\left[\octa\right]$ & $\left[\ico\right]$ & $\left[\SO(2)\right]$ & $\left[\OO(2)\right]$ \\
      \hline
      $\left[\ZZ_{m}\right]$ & $[\triv],\left[\ZZ_{d}\right]$ &                        &                       &                      &                     &                       &                       \\
      \cline{1-3}
      $\left[\DD_{m}\right]$
      &
      \begin{tabular}{c}
      $[\triv]$ \\
      $\left[ \ZZ_{d_{2}} \right]$ \\
      $\left[ \ZZ_{d} \right]$
    \end{tabular}
    &
    \begin{tabular}{c}
      $[\triv],\left[\ZZ_{2}\right],\left[\DD_{dz}\right]$                \\
      $\left[\ZZ_{d}\right],\left[\DD_{d}\right]$
    \end{tabular}
    & & & & & \\
    \cline{1-4}
    $\left[\tetra\right]$
    &
    \begin{tabular}{c}
      $[\triv]$                  \\
      $\left[\ZZ_{d_{2}}\right]$ \\
      $\left[\ZZ_{d_{3}}\right]$
    \end{tabular}
    &
    \begin{tabular}{c}
      $[\triv],\left[ \ZZ_{2}\right]$                      \\
      $\left[ \ZZ_{d_{3}}\right],\left[\DD_{d_{2}}\right]$
    \end{tabular}
    &
    \begin{tabular}{c}
      $[\triv],\left[ \ZZ_{2}\right]$               \\
      $\left[ \DD_{2}\right],\left[ \ZZ_{3}\right]$ \\
      $\left[ \tetra \right]$
    \end{tabular}
    & & & & \\
    \cline{1-5}
    $\left[\octa\right]$ &
    \begin{tabular}{c}
      $[\triv]$                  \\
      $\left[\ZZ_{d_{2}}\right]$ \\
      $\left[\ZZ_{d_{3}}\right]$ \\
      $\left[\ZZ_{d_4}\right]$
    \end{tabular}
    &
    \begin{tabular}{c}
      $[\triv],\left[\ZZ_{2}\right]$                      \\
      $\left[\ZZ_{d_{3}}\right],\left[\ZZ_{d_4}\right]$   \\
      $\left[\DD_{d_{2}}\right],\left[\DD_{d_{3}}\right]$ \\
      $\left[\DD_{d_4}\right]$
    \end{tabular}
    &
    \begin{tabular}{c}
      $[\triv]$                         \\
      $\left[ \ZZ_{2}\right],[\DD_{2}]$ \\
      $\left[ \ZZ_{3}\right]$           \\
      $\left[ \tetra \right]$
    \end{tabular}
    &
    \begin{tabular}{c}
      $[\triv],\left[ \ZZ_{2}\right]$               \\
      $\left[ \DD_{2}\right],\left[ \ZZ_{3}\right]$ \\
      $\left[ \DD_{3}\right],\left[ \ZZ_4\right]$   \\
      $\left[ \DD_4 \right],\left[ \octa \right]$
    \end{tabular}
    & & & \\
    \cline{1-6}
    $\left[\ico\right]$
    &
    \begin{tabular}{c}
      $[\triv]$                   \\
      $\left[ \ZZ_{d_{2}}\right]$ \\
      $\left[ \ZZ_{d_{3}}\right]$ \\
      $\left[ \ZZ_{d_{5}}\right]$
    \end{tabular}
    &
    \begin{tabular}{c}
      $[\triv],\left[ \ZZ_{2}\right]$                       \\
      $\left[ \ZZ_{d_{3}}\right],\left[ \ZZ_{d_{5}}\right]$ \\
      $\left[ \DD_{d_{2}}\right]$                           \\
      $\left[ \DD_{d_{3}}\right],\left[ \DD_{d_{5}}\right]$
    \end{tabular}
    &
    \begin{tabular}{c}
      $[\triv]$               \\
      $\left[ \ZZ_{2}\right]$ \\
      $\left[ \ZZ_{3}\right]$ \\
      $\left[ \tetra \right]$
    \end{tabular}
    &
    \begin{tabular}{c}
      $[\triv],\left[ \ZZ_{2}\right]$                \\
      $ \left[ \ZZ_{3}\right],\left[ \DD_{3}\right]$ \\
      $\left[ \tetra\right]$
    \end{tabular}
    &
    \begin{tabular}{c}
      $[\triv],\left[ \ZZ_{2}\right]$               \\
      $\left[ \ZZ_{3}\right],\left[ \DD_{3}\right]$ \\
      $\left[ \ZZ_{5}\right],\left[ \DD_{5}\right]$ \\
      $\left[ \ico\right]$
    \end{tabular}
    & & \\
    \cline{1-7}
    $\left[\SO(2)\right]$ &  $[\triv],\left[ \ZZ_{n}\right]$
    &
    \begin{tabular}{c}
      $[\triv]$,
      $\left[ \ZZ_{2}\right]$ \\
      $\left[ \ZZ_{n}\right]$
    \end{tabular}
    &
    \begin{tabular}{c}
      $[\triv]$ ,$\left[\ZZ_{2}\right]$ \\
      $\left[\ZZ_{3}\right]$
    \end{tabular}
    &
    \begin{tabular}{c}
      $[\triv],\left[ \ZZ_{2}\right]$               \\
      $\left[ \ZZ_{3}\right],\left[ \ZZ_{4}\right]$
    \end{tabular}
    &
    \begin{tabular}{c}
      $[\triv],\left[ \ZZ_{2}\right]$               \\
      $\left[ \ZZ_{3}\right],\left[ \ZZ_{5}\right]$
    \end{tabular}
    &  $[\triv],\left[\SO(2)\right]$ & \\
    \cline{1-8}
    $\left[\OO(2)\right]$
    &
    \begin{tabular}{c}
      $[\triv],\left[ \ZZ_{d_{2}}\right]$ \\
      $\left[ \ZZ_{n}\right]$
    \end{tabular}
    &
    \begin{tabular}{c}
      $[\triv],\left[ \ZZ_{2}\right]$               \\
      $\left[ \DD_{k_{2}}\right],\left[ \DD_{n}\right]$
    \end{tabular}
    &
    \begin{tabular}{c}
      $[\triv],\left[ \ZZ_{2}\right]$ \\
      $\left[ \DD_{2}\right]$, $\left[ \ZZ_{3}\right]$
    \end{tabular}
    &
    \begin{tabular}{c}
      $[\triv],\left[ \ZZ_{2}\right]$             \\
      $\left[ \DD_{2}\right]$, $\left[ \DD_{3}\right]$\\
      $\left[ \DD_4\right]$
    \end{tabular}
    &
    \begin{tabular}{c}
      $[\triv],\left[ \ZZ_{2}\right]$               \\
      $\left[ \DD_{2}\right]$,$\left[ \DD_{3}\right]$\\
      $\left[ \DD_{5}\right]$
    \end{tabular}
    &
    \begin{tabular}{c}
      $[\triv],\left[ \ZZ_{2}\right]$ \\
      $\left[ \SO(2)\right]$
    \end{tabular}
    &
    \begin{tabular}{c}
      $\left[ \ZZ_{2}\right],[\DD_{2}]$ \\
      $\left[ \OO(2)\right]$
    \end{tabular} \\
    \hline
    \end{tabular}
    \caption{Clips operations for $\SO(3)$}
    \label{tab:SO3-clips}
  \end{table}
\end{small}

\begin{rem}
  The clips operations $[\tetra]\circledcirc [\tetra]$ and $[\tetra]\circledcirc [\octa]$ were wrong in~\cite{OA2013,Oli2014}, since for instance the isotropy class $[\DD_{2}]$ was omitted.
\end{rem}

\begin{ex}[Isotropy classes for a family of $n$ vectors]
  For one vector, we get
  \begin{equation*}
    \mathfrak{I}(\Hn{1}) = \set{[\SO(2)],[\SO(3)]}.
  \end{equation*}
  From~\autoref{tab:SO3-clips}, we deduce that the isotropy classes for a family of $n$ vectors ($n \ge 2$) is
  \begin{equation*}
    \mathfrak{I}\left(\bigoplus_{k=1}^{n} \Hn{1}\right) = \set{[\triv],[\SO(2)],[\SO(3)]}.
  \end{equation*}
\end{ex}

\begin{ex}[Isotropy classes for a family of $n$ quadratic forms]
  The space of quadratic forms on $\RR^{3}$, $\mathrm{S}_2(\RR^{3})$, decomposes into two irreducible components (\emph{deviatoric} and \emph{spherical} tensors for the mechanicians):
  \begin{equation*}
    \mathrm{S}_2(\RR^{3}) = \Hn{2} \oplus \Hn{0}.
  \end{equation*}
  We get thus
  \begin{equation*}
    \mathfrak{I}(\mathrm{S}_2(\RR^{3})) = \mathfrak{I}(\Hn{2})=\set{[\DD_{2}],[\OO(2)],[\SO(3)]}.
  \end{equation*}
  The useful part of~\autoref{tab:SO3-clips}, for our purpose, reads:
  \begin{center}
    \begin{tabular}{|c|c|c|c|}
      \hline
      $\circledcirc$ & $[\ZZ_{2}]$               & $[\DD_{2}]$                         & $[\OO(2)]$                           \\
      \hline
      $[\ZZ_{2}]$    & $\set{[\triv],[\ZZ_{2}]}$ & $\set{[\triv],[\ZZ_{2}]}$           & $\set{[\triv],[\ZZ_{2}]}$            \\
      \hline
      $[\DD_{2}]$    &                           & $\set{[\triv],[\ZZ_{2}],[\DD_{2}]}$ & $\set{[\triv],[\ZZ_{2}],[\DD_{2}]}$  \\	
      \hline
      $[\OO(2)]$     &                           &                                     & $\set{[\ZZ_{2}],[\DD_{2}],[\OO(2)]}$ \\
      \hline
    \end{tabular}
  \end{center}
  We deduce therefore that the set of isotropy classes for a family of $n$ quadratic forms ($n \ge 2$) is
  \begin{equation*}
    \mathfrak{I}\left( \bigoplus_{k=1}^{n} \mathrm{S}_2(\RR^{3}) \right) = \set{[\triv],[\ZZ_{2}],[\DD_{2}],[\OO(2)],[\SO(3)]}.
  \end{equation*}
\end{ex}

\subsection{$\OO(3)$ closed subgroups}

Let us first consider an $\OO(3)$-representation where $-\Idd$ act as $-\mathrm{Id}$ (meaning that this representation doesn't reduce to some $\SO(3)$ representation). In such a case, only the null vector can be fixed by $-\mathrm{Id}$, and so type II subgroups never appear as isotropy subgroups. In that case, we need only to focus on clips operations between type I and type III subgroups, and then between type III subgroups, since clips operations between type I subgroups have already been considered in~\autoref{tab:SO3-clips}. For type III subgroups as detailed in~\autoref{sec:proofs-O3} we have:

\begin{lem}\label{lem:typeIetIII}
  Let $H_{1}$ be some type III closed subgroup of $\OO(3)$ and $H_{2}$ be some type I closed subgroup of $\OO(3)$. Then we have
  \begin{equation*}
    H_{1}\cap H_{2}=(H_{1}\cap \SO(3))\cap H_{2},
  \end{equation*}
  and for every closed subgroup $H$ of $\SO(3)$, we get:
  \begin{align*}
    [\ZZ_{2}^{-}]\circledcirc [H] & =\set{[\triv]},             & [\ZZ_{2n}^{-}]\circledcirc [H] & =[\ZZ_{n}]\circledcirc [H], \\
    [\DD_{n}^{v}]\circledcirc [H] & =[\ZZ_{n}]\circledcirc [H], & [\DD_{2n}^h]\circledcirc [H]   & =[\DD_{n}]\circledcirc [H], \\
    [\octa^{-}]\circledcirc [H]   & =[\tetra]\circledcirc [H],  & [\OO(2)^{-}]\circledcirc [H]   & =[\SO(2)]\circledcirc [H].
  \end{align*}
\end{lem}

The resulting conjugacy classes for the clips operation for type III subgroups are given in ~\autoref{tab:otrglobal}, where the following notations have been used:
\begin{align*}
  d        & := \gcd(n,m), & d_{2}(n) & := \gcd(n,2),   \\
  d_{3}(n) & := \gcd(n,3), & i(n)     & := 3-\gcd(2,n), \\
  \ZZ_{1}^-&=\DD_{1}^v=\DD_{2}^h=\triv.
\end{align*}

\begin{scriptsize}
  \begin{table}[h]
    \renewcommand{\arraystretch}{1.8}
    \renewcommand\tabcolsep{2pt}
    \begin{tabular}{|c|c|c|c|c|c|c|}
      \hline
      $\circledcirc$            & $\left[\ZZ_{2}^-\right]$                & $\left[\ZZ_{2m}^-\right]$          & $\left[\DD_{m}^{v}\right]$ & $\left[\DD_{2m}^h\right]$ & $\left[\octa^-\right]$ & $\left[\OO(2)^-\right]$ \\
      \hline
      $\left[\ZZ_{2}^-\right]$  & $[\triv],\left[\ZZ_{2}^-\right]$        &                                    &                            &                           &                        &                         \\
      \cline{1-3}
      $\left[\ZZ_{2n}^-\right]$ & $[\triv],\left[ \ZZ_{i(n)}^{-}\right] $ & Figure \autoref{fig:clips-Z2n-Z2m} &                            &                           &                        &                         \\
      \cline{1-4}
      $\left[\DD_{n}^{v}\right]$ & $[\triv],\left[\ZZ_{2}^-\right]$ & Figure \autoref{fig:clips-Dn-Z2m}
      &
      \begin{tabular}{c}
      $[\triv],\left[ \ZZ_{2}^{-}\right]$ \\
      $\left[ \DD_{d}^{v} \right],\left[ \ZZ_{d} \right]$
    \end{tabular}
    & & & \\
    \cline{1-5}
    $\left[\DD_{2n}^h\right]$ & $[\triv],\left[\ZZ_{2}^-\right]$ & \autoref{fig:clips-D2n-Z2m} &  \autoref{fig:clips-D2n-Dm}  &  \autoref{fig:clips-D2n-D2m} & &
    \\
    \cline{1-6}
    $\left[\octa^-\right]$ & $[\triv],\left[\ZZ_{2}^-\right]$ &  Figure \autoref{fig:clips-cubic-Z2m}
    &
    \begin{tabular}{c}
      $[\triv],[\ZZ_{2}^{-}]$ \\
      $[\ZZ_{d_{3}(m)}]$      \\
      $[\DD_{d_{3}(m)}^{v}]$  \\
      $[\ZZ_{d_{2}(m)}]$      \\
      $[\DD_{d_{2}(m)}^{v}]$
    \end{tabular}
    &  Figure \autoref{fig:clips-cubic-D2m}
    &
    \begin{tabular}{l}
      $[\triv],[\ZZ_{2}^{-}]$ \\
      $[\ZZ_4^{-}],[\ZZ_{3}]$
    \end{tabular}
    & \\
    \cline{1-7}
    $\left[\OO(2)^-\right]$ & $[\triv],\left[\ZZ_{2}^-\right]$
    &
    \begin{tabular}{c}
      $[\triv],[\ZZ_{i(m)}^{-}]$ \\
      $[\ZZ_{m}]$
    \end{tabular}
    &
    $[\triv],[\ZZ_{2}^{-}],[\DD_{m}^{v}]$
    &
    \begin{tabular}{c}
      $[\triv]$                         \\
      $[\ZZ_{d_{2}(m)}],[\ZZ_{2}^{-}]$  \\
      $ [\DD_{i(m)}^{v}],[\DD_{m}^{v}]$
    \end{tabular}
    &
    \begin{tabular}{c}
      $[\triv],[\ZZ_{2}^{-}]$       \\
      $[\DD_{3}^{v}],[\DD_{2}^{v}]$
    \end{tabular}
    &
    $[\ZZ_{2}^{-}],[\OO(2)^-]$ \\
    \hline
    \end{tabular}
    \caption{Clips operations on type III $\OO(3)$-subgroups}
    \label{tab:otrglobal}
  \end{table}
\end{scriptsize}

\begin{rem}
  One misprint in~\cite{Oli2014,OA2014} for clips operation $[\OO(2)^{-}]\circledcirc [\DD_{2m}^h]$ has been corrected: the conjugacy class $[\ZZ_{2}]$ appears for $m$ even (and not for $m$ odd).
\end{rem}

\begin{figure}[H]
  \centering
  \subfigure[$\ZZ_{2n}^{-}$ case]{
    \includegraphics[scale=1]{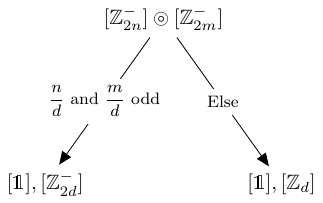}
    \label{fig:clips-Z2n-Z2m}
  }
  \subfigure[$\DD_{n}^{v}$ and $\ZZ_{2n}^{-}$]{
    \includegraphics[scale=0.9]{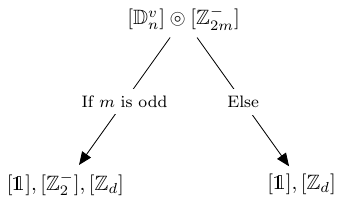}
    \label{fig:clips-Dn-Z2m}
  }
\end{figure}

\begin{figure}[H]
  \centering
  \includegraphics[scale=1]{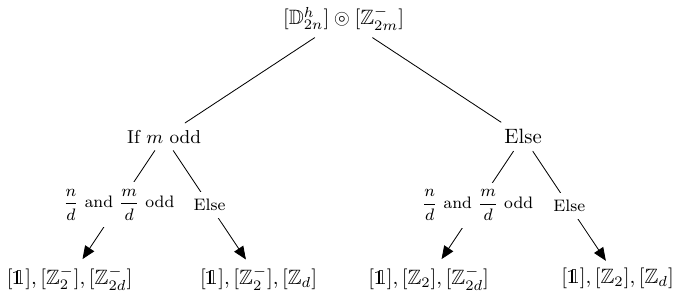}
  \caption{$\DD_{2n}^h$ and $\ZZ_{2n}^{-}$}
  \label{fig:clips-D2n-Z2m}
\end{figure}

\begin{figure}[H]
  \centering
  \includegraphics[scale=1]{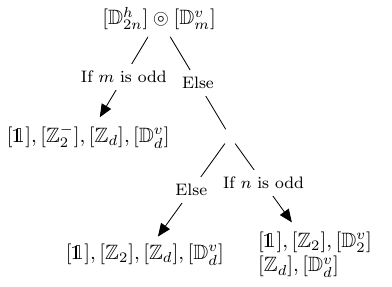}
  \caption{$\DD_{2n}^h$ and $\DD_{n}^{v}$}
  \label{fig:clips-D2n-Dm}
\end{figure}

\begin{figure}[H]
  \centering
  \includegraphics[scale=0.84]{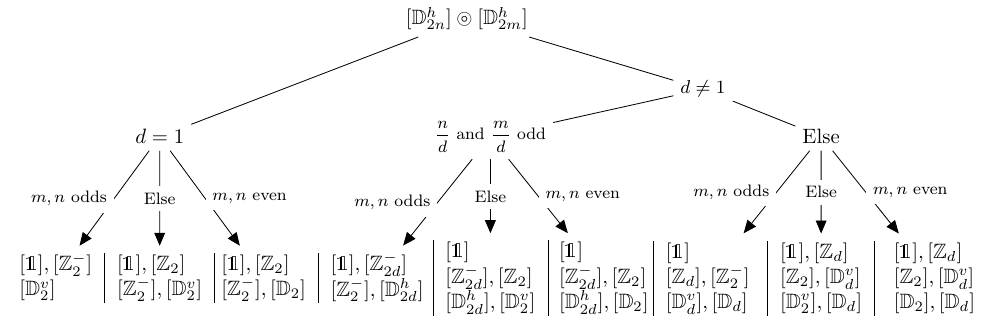}
  \caption{$\DD_{2n}^h$ and $\DD_{2m}^{h}$}
  \label{fig:clips-D2n-D2m}
\end{figure}

\begin{figure}[H]
  \centering
  \subfigure[$\OO^-$ and $\ZZ_{2m}^{-}$]{
    \includegraphics[scale=0.85]{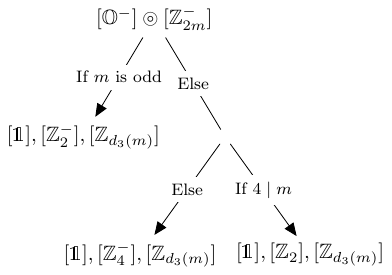}
    \label{fig:clips-cubic-Z2m}}
  \subfigure[$\OO^-$ and $\DD_{2m}^{h}$]{
    \includegraphics[scale=0.85]{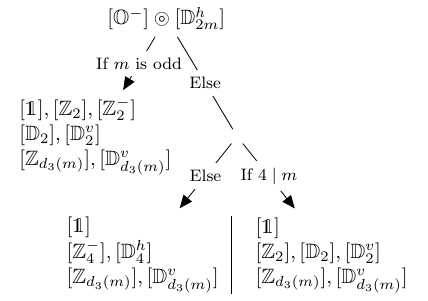}
    \label{fig:clips-cubic-D2m}
  }
\end{figure}

\subsection{Application to tensorial mechanical properties}\label{subsec:Mechanical_Prop}

We propose here some direct applications of our results for many different tensorial spaces, each one being endowed with the natural $\OO(3)$ representation. Such tensorial spaces occur for instance in the modeling of mechanical properties. The main idea is to use the \emph{irreducible decomposition}, also known as the \emph{harmonic decomposition}~\cite{Bac1970,FV1996}. We then use clips operations given in Table~\ref{tab:SO3-clips} and Table~\ref{tab:otrglobal}. Note that we don't need to know explicit irreducible decomposition of those tensorial representations.

From now on, we define $\Hn{n}^*$ to be $\OO(3)$ representation given by $(\Hn{n},\rho_{n}^*)$. Furthermore, we define 
\begin{equation*}
	\Hn{n}^{\oplus k}:=\bigoplus_{i=1}^{k} \Hn{n},\quad 	\Hn{n}^{\oplus k*}:=\bigoplus_{i=1}^{k} \Hn{n}^*.
\end{equation*}
We propose here to give some specific irreducible decompositions, related to mechanical theory, without any further details. We explain the use of clips operation in the classical case of Elasticity, all other examples being done in the same way.
 
\subsubsection{Classical results}

We first give some classsical results we obtain here directly

(1) Elasticity~\cite{FV1996}: $\SO(3)$ tensor space
\begin{equation*}
	\mathbb{E}\mathrm{la}\simeq \Hn{4}\oplus \Hn{2}^{\oplus 2}\oplus \Hn{0}^{\oplus 2}.
\end{equation*}
We have from Theorem~\ref{thm:SO3_Irr_Isot}
\begin{align*}
	\mathfrak{I}(\Hn{2})&=\{[\DD_2],[\OO(2)],[\SO(3)]\}, \\
	\mathfrak{I}(\Hn{4})&=\{[\triv],[\ZZ_2],[\DD_2],[\DD_3],[\DD_4],[\octa],[\OO(2)],[\SO(3)]\}
\end{align*}
so, to obtain $\mathfrak{I}(\Hn{2}^{\oplus 2})$ we use the clips table 
  \begin{center}
    \begin{tabular}{|c|c|c|}
      \hline
      $\circledcirc$ & $[\DD_{2}]$                         & $[\OO(2)]$                           \\
      \hline
      $[\DD_{2}]$    & $\set{[\triv],[\ZZ_{2}],[\DD_{2}]}$ & $\set{[\triv],[\ZZ_{2}],[\DD_{2}]}$  \\	
      \hline
      $[\OO(2)]$     &                                     & $\set{[\ZZ_{2}],[\DD_{2}],[\OO(2)]}$ \\
      \hline
    \end{tabular}
  \end{center}
so that
\begin{align*}
	\mathfrak{I}(\Hn{2}^{\oplus 2})&=\mathfrak{I}(\Hn{2})\circledcirc \mathfrak{I}(\Hn{2}) \\
	&=\left( [\DD_2]\circledcirc [\DD_2]\right) \bigcup \left( [\DD_2]\circledcirc [\OO_2]\right) \bigcup
	\left( [\OO_2]\circledcirc [\OO_2]\right) \\
	&=\{ [\triv],[\ZZ_{2}],[\DD_{2}],[\OO(2)]\}
\end{align*}
and now we conclude using the clips table coming from the clips operation
\begin{equation*}
	\mathfrak{I}(\mathbb{E}\mathrm{la})=\mathfrak{I}(\Hn{4})\circledcirc \mathfrak{I}(\Hn{2}^{\oplus 2}).
\end{equation*}
\noindent Finally we obtain 8 symmetry classes
\begin{equation*}
	\mathfrak{I}(\mathbb{E}\mathrm{la})=\{[\triv],[\ZZ_2],[\DD_2],[\DD_3],[\DD_4],[\octa],[\OO(2)],[\SO(3)]\}.
\end{equation*}

(2) Photoelasticity~\cite{FV1997}: $\SO(3)$ tensor space 
\begin{align*}
	\mathbb{P}\mathrm{la}\simeq \Hn{4}\oplus &\Hn{3}\oplus \Hn{2}^{\oplus 3}\oplus \Hn{1}\oplus \Hn{0}^{\oplus 2}.
\end{align*}
\noindent 12 symmetry classes
\begin{multline*}
	\mathfrak{I}(\mathbb{P}\mathrm{la})=\lbrace [\triv],[\ZZ_2],[\DD_2],[\ZZ_3],[\DD_3],[\ZZ_4],[\DD_4],[\tetra],[\octa],\\
	[\SO(2)],[\OO(2)],[\SO(3)]\rbrace.
\end{multline*}

(3) Piezoelecricity~\cite{GW2002}: $\OO(3)$ tensor space 
\begin{align*}
	\Pie\simeq\Hn{3}\oplus \Hn{2}^*\oplus \Hn{1}^{\oplus 2}.
\end{align*}
\noindent 16 symmetry classes
\begin{multline*}
\mathfrak{I}(\mathbb{P}\mathrm{iez})=\lbrace [\triv],[\ZZ_2],[\ZZ_{3}],[\DD^{v}_2],[\DD^{v}_{3}],[\ZZ^{-}_2],[\ZZ^{-}_{4}],[\DD_2],[\DD_{3}],[\DD^{h}_4],[\DD^{h}_{6}],\\
[\SO(2)],[\OO(2)],[\OO(2)^{-}],[\octa^-],[\OO(3)]\rbrace.
\end{multline*}

\begin{rem}
Note that the symmetry classes of the space $(\Pie,\OO(3))$ appear in many different works : in Weller phD thesis~\cite{Wel2004}[Theorem 3.19, p.84], where 14 symmetry classes are announced, but 15 appeared in the poset, in Weller--Geymonat~\cite{GW2002} where they establish 14 symmetry classes, in Olive--Auffray~\cite{OA2014} (with a typo), in Zou \& al~\cite{ZTP2013} (without $[\OO(3)]$ symmetry class), and finally in Olive~\cite{Oli2014}, with 16 symmetry classes. 
\end{rem}

\subsubsection{Non classical results}

We now present some non classical tensor space, coming from Cosserat elasticity~\cite{Cos1909,Eri1966,For2005,FS2006} and Strain gradient elasticity~\cite{Min1964,ME1968,PB1968,ADR2015}, with their harmonic decomposition and symmetry classes.

(1) Classical Cosserat elasticity: $\SO(3)$ tensor space  
\begin{align*}
	\mathbb{C}\mathrm{os}\simeq \Hn{4}\oplus &\Hn{3}\oplus \Hn{2}^{\oplus 4}\oplus\Hn{1}^{\oplus 2}\oplus \Hn{0}^{\oplus 3}.
\end{align*}
\noindent 12 symmetry classes
\begin{multline*}
	\mathfrak{I}(\mathbb{C}\mathrm{os})=\lbrace [\triv],[\ZZ_2],[\ZZ_3],[\ZZ_4],[\DD_2],[\DD_3],[\DD_4],[\tetra],[\octa],\\
	[\SO(2)],[\OO(2)],[\SO(3)]\rbrace.
\end{multline*}

(2) Rotational Cosserat elasticity: $\OO(3)$ tensor space
\begin{align*}
	\mathbb{C}\mathrm{hi}\simeq \Hn{4}^*\oplus &\Hn{3}^{\oplus 3}\oplus \Hn{2}^{\oplus 6*}\oplus\Hn{1}^{\oplus 6}\oplus \Hn{0}^{\oplus 3*}.
\end{align*}
\noindent 24 symmetry classes 
\begin{multline*}
	\mathfrak{I}(\mathbb{C}\mathrm{hi})=\lbrace [\triv],[\ZZ_2],[\ZZ_3],[\ZZ_4],[\ZZ_2^-],[\ZZ_4^-],[\ZZ_6^-],[\DD_2],[\DD_3],[\DD_4],\\
	[\DD_2^v],[\DD_3^v],[\DD_4^v],[\DD_4^h],[\DD_6^h],[\DD_8^h],[\tetra],[\octa],[\octa^-],[\SO(2)],\\
	[\OO(2)],[\OO(2)^-],[\SO(3)],[\OO(3)]\rbrace.
\end{multline*}

(3) Fifth-order $\OO(3)$ tensor space in strain gradient elasticity, given by
\begin{align*}
	\mathbb{S}\mathrm{ge}\simeq \Hn{5}\oplus &\Hn{4}^{\oplus 2*}\oplus \Hn{3}^{\oplus 5}\oplus\Hn{2}^{\oplus 5*}\oplus \Hn{1}^{\oplus 6}\oplus \Hn{0}^{*}.
\end{align*}
\noindent 29 symmetry classes
\begin{multline*}
	\mathfrak{I}(\mathbb{S}\mathrm{ge})=\lbrace [\triv],[\ZZ_2],[\ZZ_3],[\ZZ_4],[\ZZ_5],[\ZZ_2^-],[\ZZ_4^-],\\
	[\ZZ_6^-],[\ZZ_8^-],[\DD_2],[\DD_3],[\DD_4],[\DD_5],[\DD_2^v],[\DD_3^v],[\DD_4^v],\\
	[\DD_5^v],[\DD_4^h],[\DD_6^h],[\DD_8^h],[\DD_{10}^h],[\tetra],[\octa],[\octa^-],\\
	[\SO(2)],[\OO(2)],[\OO(2)^-],[\SO(3)],[\OO(3)]\rbrace.
\end{multline*}
(4) Fifth order $\OO(3)$ tensor space of acoustical gyrotropic tensor~\cite{PB1968} (reducing to fourth order tensor space), given by
\begin{align*}
	\mathbb{A}\mathrm{gy}\simeq \Hn{4}^{*}\oplus &\Hn{3}^{\oplus 2}\oplus \Hn{2}^{\oplus 3*}\oplus\Hn{1}^{\oplus 2}\oplus \Hn{0}^{*}.
\end{align*}
\noindent 24 symmetry classes
\begin{multline*}
	\mathfrak{I}(\mathbb{A}\mathrm{gy})=\lbrace [\triv],[\ZZ_2],[\ZZ_3],[\ZZ_4],[\ZZ_2^-],[\ZZ_4^-],[\ZZ_6^-],[\DD_2],[\DD_3],[\DD_4],\\
	[\DD_2^v],[\DD_3^v],[\DD_4^v],[\DD_4^h],[\DD_6^h],[\DD_8^h],[\tetra],[\octa],[\octa^-],[\SO(2)],\\
	[\OO(2)],[\OO(2)^-],[\SO(3)],[\OO(3)]\rbrace.
\end{multline*}

\appendix

\section{Proofs for $\SO(3)$}
\label{sec:proofs-SO3}

In this section, we provide all the details required to obtain the results in~\autoref{tab:SO3-clips}. We will start by the following definition which was introduced in~\cite{GSS1988} and happens to be quite useful for this task.
\begin{defn}
  Let $K_{1}, K_{2}, \dotsc , K_s$ be subgroups of a given group $G$. We say that $G$ is the \emph{direct union} of the $K_{i}$ and we write $G=\biguplus_{i=1}^s K_{i}$ if
  \begin{equation*}
    G = \bigcup_{i=1}^s K_{i} \qquad \text{and} \qquad K_{i}\cap K_{j} = \set{e},\quad \forall i\neq j.
  \end{equation*}
\end{defn}

In the following, we will have to identify repeatedly the conjugacy class of intersections such as
\begin{equation}\label{eq:Int_Subgroup}
  H_{1}\cap \left(g H_{2}g^{-1}\right),
\end{equation}
where $H_{1}$ and $H_{2}$ are two closed subgroups of $\SO(3)$ and $g\in \SO(3)$. A useful observation is that all closed $\SO(3)$ subgroups have some \emph{characteristic axes} and that intersection~\eqref{eq:Int_Subgroup} depends only on the relative positions of these characteristic axes.

As detailed below, for any subgroup conjugate to $\ZZ_{n}$ or $\DD_{n}$ ($n\geq 3$), the axis of an $n$-th order rotation (in this subgroup) is called its \emph{primary axis}. For subgroups conjugate to $\DD_{n}$ ($n\geq 3$), axes of order two rotations are said to be \emph{secondary axes}. In the special case $n=2$, the $z$-axis is the primary axis of $\ZZ_{2}$, while any of the $x$, $y$ or $z$ axis can be considered as a primary axis of $\DD_{2}$.

\subsection{Cyclic subgroup}

For any axis $a$ of $\RR^{3}$ (throughout the origin), we denote by $\ZZ_{n}^{a}$, the unique cyclic subgroup of order $n$ around the $a$-axis, which is its primary axis. We have then:

\begin{lem}\label{lem:intercy}
  Let $m,n\geq 2$ be two integers and $d=\gcd(n,m)$. Then
  \begin{equation*}
    [\ZZ_{n}]\circledcirc [\ZZ_{m}]=\set{[\triv],[\ZZ_{d}]}.
  \end{equation*}
\end{lem}

\begin{proof}
  We have to consider intersections, such as
  \begin{equation*}
    \ZZ_{n}\cap (g\ZZ_{m}g^{-1})=\ZZ_{n}\cap \ZZ_{m}^{a},
  \end{equation*}
  for some axis $a$, and only two cases occur:
  \begin{enumerate}
    \item If $a\neq (Oz)$, then necessarily the intersection reduces to $\triv$.
    \item If $a=(Oz)$, then the order $r$ of a rotation in $\ZZ_{n}\cap \ZZ_{m}$ divides both $n$ and $m$ and thus divides $d=\gcd(m,n)$. We get therefore: $\ZZ_{n}\cap \ZZ_{m} \subset \ZZ_{d}$. But obviously, $\ZZ_{d} \subset \ZZ_{n}\cap \ZZ_{m}$ and thus $\ZZ_{n}\cap \ZZ_{m}=\ZZ_{d}$.
  \end{enumerate}
\end{proof}

\subsection{Dihedral subgroup}

Let $b_{1}$ be the $x$-axis and $b_{k}$ ($k=2,\dotsc,n$) be the axis recursively defined by
\begin{equation*}
  b_{k} := \QQ\left(\kk;\displaystyle{\frac{\pi}{n}}\right)b_{k-1}.
\end{equation*}
Then, we have
\begin{equation}\label{diedre}
  \DD_{n} = \ZZ_{n} \biguplus \ZZ_{2}^{b_{1}} \biguplus \dotsb \biguplus \ZZ_{2}^{b_{n}},
\end{equation}
where the $z$-axis (corresponding to a $n$-th order rotation) is the primary axis and the $b_{k}$-axes (corresponding to order two rotations) are the secondary axes of this dihedral group (see~\autoref{fig:dihedral-second-axis}).
\begin{figure}[h]
  \centering
  \includegraphics{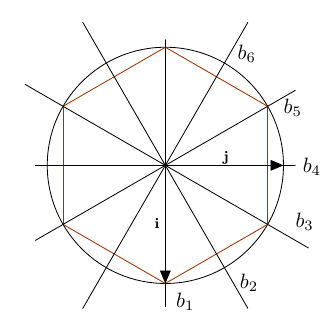}
  \caption{Secondary axis of the dihedral group $\DD_{6}$}
  \label{fig:dihedral-second-axis}
\end{figure}

\begin{lem}\label{lem:Clis_Diedre_Cyclic}
  Let $m,n\geq 2$ be two integers. Set $d:=\gcd(n,m)$ and $d_{2}(m):=\gcd(m,2)$. Then, we have
  \begin{equation*}
    [\DD_{n}]\circledcirc [\ZZ_{m}]=\set{[\triv],[\ZZ_{d_{2}(m)}],[\ZZ_{d}] }.
  \end{equation*}
\end{lem}

\begin{proof}
  Let $\Gamma=\DD_{n} \cap g\ZZ_{m}g^{-1}$ for $g\in \SO(3)$. From decomposition \eqref{diedre}, we have to consider intersections
  \begin{equation*}
    \ZZ_{n}\cap g\ZZ_{m}g^{-1},\quad \ZZ_{2}^{b_{j}}\cap g\ZZ_{m}g^{-1}.
  \end{equation*}
  which thus reduce to Lemma~\ref{lem:intercy}.
\end{proof}

\begin{lem}
  Let $m,n\geq 2$ be two integers. Set $d := \gcd(n,m)$ and
  \begin{equation*}
   dz :=
    \begin{cases}
      2 \quad \text{if} \quad m \text{ and } n \text{ even}, \\
      1 \quad \text{otherwise}.
    \end{cases}
  \end{equation*}
  Then, we have
  \begin{equation*}
    [\DD_{n}]\circledcirc [\DD_{m}] = \set{[\triv],[\ZZ_{2}],[\DD_{dz}],[\ZZ_{d}],[\DD_{d}]}.
  \end{equation*}
\end{lem}

\begin{proof}
  Let $\Gamma=\DD_{n} \cap \left(g\DD_{m} g^{-1}\right)$.
  \begin{enumerate}
    \item If both primary axes and one secondary axis match, $\Gamma=\DD_{d}$ if $d\neq 1$ and $\Gamma\in [\ZZ_{2}]$ otherwise;
    \item if only the primary axes match, $\Gamma=\ZZ_{d}$;
    \item if the angle of primary axes is $\dfrac{\pi}{4}$ and a secondary axis match, then $\Gamma\in [\ZZ_{2}]$
    \item if the primary axis of $g\DD_{m}g^{-1}$ matches with the secondary axis $(Ox)$ of $\DD_{n}$ (or the converse), we obtain $\Gamma\in [\DD_{2}]$ for $n$ and $m$ even and a secondary axis of $g\DD_{m}g^{-1}$ is $(Oz)$, otherwise we obtain $\Gamma\in [\ZZ_{2}]$
    \item in all other cases we have $\Gamma=\triv$.
  \end{enumerate}
\end{proof}

\subsection{Tetrahedral subgroup}

The (orientation-preserving) symmetry group $\tetra$ of the tetrahedron $\mathcal{T}_{0}:=A_{1}A_{3}A_7A_{5}$ (see \autoref{fig:cube0}) decomposes as (see~\cite{IG1984}):
\begin{equation}\label{dtetra}
  \tetra= \biguplus_{i=1}^4 \ZZ_{3}^{\avt_{i}} \biguplus_{j=1}^3 \ZZ_{2}^{\aet_{j}}
\end{equation}
where $\avt_{i}$ (resp. $\aet_{j}$) are the \emph{vertices axes} (resp. \emph{edges axes}) of the tetrahedron (see ~\autoref{fig:cube0}):
\begin{align*}
  \avt_{1} & := (OA_{1}), & \avt_{2} & := (OA_{3}), & \avt_{3} & := (OA_{5}), & \avt_4 & := (OA_7),
  \\
  \aet_{1} & := (Ox),     & \aet_{2} & :=(Oy) ,     & \aet_{3} & :=(Oz).      &        &
\end{align*}

\begin{figure}[h]
  \centering
  \includegraphics{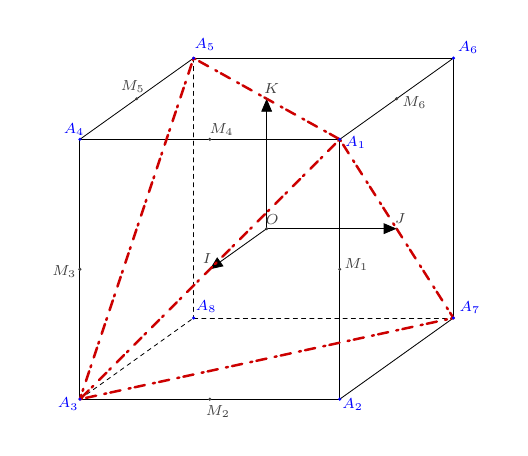}
  \caption{Cube $\mathcal{C}_{0}$ and tetrahedron $\mathcal{T}_{0}:=A_{1}A_{3}A_7A_{5}$}
  \label{fig:cube0}
\end{figure}

\begin{cor}\label{lem:Clips_Tetra_Cyclic}
  Let $n\geq 2$ be an integer. Set $d_{2}(n):=\gcd(n,2)$ and $d_{3}(n):=\gcd(3,n)$. Then, we have
  \begin{equation*}
    [\ZZ_{n}] \circledcirc [\tetra] = \set{[\triv],[\ZZ_{d_{2}(n)}],[\ZZ_{d_{3}(n)}]}.
  \end{equation*}
\end{cor}

\begin{proof}
  Consider $\tetra \cap \ZZ_{n}^{a}$ for some axis $a$. As a consequence of Lemma~\ref{lem:intercy}, we need only to consider the case where $a$ is an edge axis or a face axis of the tetrahedron, reducing to the clips operations
  \begin{equation*}
    [\ZZ_{2}]\circledcirc [\ZZ_{n}],\quad [\ZZ_{3}]\circledcirc [\ZZ_{n}]
  \end{equation*}
  which directly leads to the Lemma.
\end{proof}

\begin{cor}
  Let $n\geq 2$ be some integer. Set $d_{2}(n):=\gcd(n,2)$ and $d_{3}(n):=\gcd(3,n)$. Then, we have
  \begin{equation*}
    [\DD_{n}] \circledcirc [\tetra] = \set{[\triv],[\ZZ_{2}],[\ZZ_{d_{3}(n)}],[\DD_{d_{2}(n)}]}.
  \end{equation*}
\end{cor}

\begin{proof}
  Let $\Gamma=\tetra\cap \left(g\DD_{n} g^{-1}\right)$. From decomposition~\eqref{dtetra}, we need only to consider intersections
  \begin{equation*}
    \ZZ_{3}^{\avt_{i}}\cap \left(g\DD_{n} g^{-1}\right) \text{ and } \ZZ_{2}^{\aet_{j}}\cap \left(g\DD_{n} g^{-1}\right)
  \end{equation*}
  which have already been studied (see Lemma~\ref{lem:Clis_Diedre_Cyclic}).
\end{proof}

\begin{lem}\label{lem:Clips_Tetra_Tetra}
  We have
  \begin{equation*}
    [\tetra]\circledcirc [\tetra] = \set{[\triv],[\ZZ_{2}],[\DD_{2}],[\ZZ_{3}],[\tetra]}.
  \end{equation*}
\end{lem}

\begin{proof}
  Let $\Gamma=\tetra\cap \left(g\tetra g^{-1}\right)$.
  \begin{enumerate}
    \item If no axes match, then $\Gamma=\triv$;
    \item if only one edge axis (resp. one face axis) from both configurations match, then $\Gamma\in [\ZZ_{2}]$ (resp. $[\ZZ_{3}]$);
    \item if $g=\QQ\left(\kk,\frac{\pi}{2}\right)$, then $\Gamma=\DD_{2}$.
  \end{enumerate}
\end{proof}

\subsection{Octahedral subgroup}

The group $\octa$ is the (orientation-preserving) symmetry group of the cube $\mathcal{C}_{0}$ (see~\autoref{fig:cube0}) with vertices
\begin{equation*}
  \{A_{i}\}_{i=1\cdots 8} = {(\pm1,\pm1,\pm1)}.
\end{equation*}
We have the decomposition (see~\cite{IG1984}):
\begin{equation}\label{eq:deccube}
  \octa=\biguplus_{i=1}^3 \ZZ_4^{\afc_{i}} \biguplus_{j=1}^4 \ZZ_{3}^{\avc_{j}} \biguplus_{l=1}^6 \ZZ_{2}^{\aec_{l}}
\end{equation}
with \emph{vertices}, \emph{edges} and \emph{faces axes} respectively denoted $\avc_{i}$, $\aec_{j}$ and $\afc_{j}$. For instance we have
\begin{equation*}
  \avc_{1} := (OA_{1}),\quad \aec_{1} := (OM_{1}),\quad \afc_{1} := (OI).
\end{equation*}

As an application of decomposition~\eqref{eq:deccube} and Lemma~\ref{lem:intercy}, we obtain the following corollary.

\begin{cor}
  Let $n\geq 2$ be some integer. Set
  \begin{equation*}
    d_{2}(n) = \gcd(n,2), \quad d_{3}(n) = \gcd(n,3),
  \end{equation*}
  and
  \begin{equation*}
    d_4(n)=
    \begin{cases}
      4 \text{ if } 4 \text{ divide } n, \\
      1 \text{ otherwise}.
    \end{cases}
  \end{equation*}
  Then, we have
  \begin{equation*}
    [\ZZ_{n}]\circledcirc [\octa] = \set{[\triv],[\ZZ_{d_{2}(n)}],[\ZZ_{d_{3}(n)}],[\ZZ_{d_4(n)}]}.
  \end{equation*}
\end{cor}

\begin{cor}
  Let $n\geq 2$ be some integer. Set
  \begin{equation*}
    d_{2}(n) = \gcd(n,2), \quad d_{3}(n) = \gcd(n,3),
  \end{equation*}
  and
  \begin{equation*}
    d_4(n)=
    \begin{cases}
      4 \text{ if } 4\mid n, \\
      1 \text{ otherwise}.
    \end{cases}
  \end{equation*}
  Then, we have
  \begin{equation*}
    [\DD_{n}]\circledcirc[\octa] = \set{[\triv],[\ZZ_{2}],[\ZZ_{d_{3}(n)}],[\ZZ_{d_4(n)}],[\DD_{d_{2}(n)}],[\DD_{d_{3}(n)}],[\DD_{d_4(n)}]}.
  \end{equation*}
\end{cor}

\begin{proof}
  Using decomposition~\eqref{eq:deccube}, we have to consider intersections
  \begin{equation*}
    \DD_{n}\cap \left(g\ZZ_4^{\afc_{i}}g^{-1}\right),\quad \DD_{n}\cap \left(g\ZZ_{3}^{\avc_{j}}g^{-1}\right),\quad\DD_{n}\cap \left(g\ZZ_{2}^{\aec_{l}}g^{-1}\right)
  \end{equation*}	
  which have already been studied in Lemma~\ref{lem:Clis_Diedre_Cyclic}.
\end{proof}

\begin{lem}\label{lem:Clips_Octa_Tetra}
  We have
  \begin{equation*}
    [\tetra]\circledcirc [\octa]=\set{[\triv],[\ZZ_{2}],[\DD_{2}],[\ZZ_{3}],[\tetra]}.
  \end{equation*}
\end{lem}

\begin{proof}
  Let $\Gamma=\octa\cap \left(g\tetra g^{-1}\right)$. From decompositions~\eqref{dtetra}--\eqref{eq:deccube} and Lemma~\ref{lem:intercy}, we only have to consider intersections
  \begin{equation*}
    \ZZ_{4}^{\afc_{i}}\cap \left(g\ZZ_{2}^{\aet_{j}}g^{-1}\right),\quad \ZZ_{3}^{\avc_{j}}\cap \left(g\ZZ_{3}^{\avt_{i}}g^{-1}\right),\quad \ZZ_{2}^{\aec_{l}}\cap \left(g\ZZ_{2}^{\aet_{j}}g^{-1}\right).
  \end{equation*}
  Now, we always can find $g$ such that the intersection $\Gamma$ reduces to some subgroup conjugate to $\triv,\ZZ_{2}$ or $\ZZ_{3}$ and taking $g=\QQ\left(\kk,\frac{\pi}{4}\right)$, we get that $\Gamma$ is conjugate to $\DD_{2}$, which achieves the proof.
\end{proof}

\begin{lem}
  We have
  \begin{equation*}
    [\octa]\circledcirc [\octa]=\set{[\triv],[\ZZ_{2}],[\DD_{2}],[\ZZ_{3}],[\DD_{3}],[\ZZ_4],[\DD_4],[\octa]}.
  \end{equation*}
\end{lem}

\begin{proof}
  Consider the subgroup $\Gamma=\octa\cap \left(g\octa g^{-1}\right)\subset \octa$. From the poset in \autoref{fig:SO3-lattice}, we deduce that the conjugacy class $[\Gamma]$ belong to the following list
  \begin{equation*}
    \set{[\triv],[\ZZ_{2}],[\DD_{2}],[\ZZ_{3}],[\DD_{3}],[\ZZ_4],[\DD_4],[\tetra],[\octa] }.
  \end{equation*}
  Now:
  \begin{enumerate}
    \item if $g$ fixes only one edge axis (resp. one vertex axis), then $\Gamma\in[\ZZ_{2}]$ (resp. $\Gamma\in[\ZZ_{3}]$);
    \item if $g=\displaystyle{\QQ\left( \ii;\frac{\pi}{6}\right)}$, only one face axis is fixed by $g$ and $\Gamma\in[\ZZ_4]$;
    \item if $g=\displaystyle{\QQ\left( \ii;\frac{\pi}{4}\right)}$, $\Gamma \supset\ZZ_4^{\ii}\uplus \ZZ_{2}^{\kk}$ and thus $\Gamma \in[\DD_4]$;
    \item if $g=\displaystyle{\QQ\left( \kk;\frac{\pi}{4}\right) \circ \QQ\left( \ii;\frac{\pi}{4}\right)}$, $\Gamma\in[\DD_{2}]$ with characteristic axes $g\afc_{3}=\aec_6$, $g\aec_{1}=\afc_{1}$ and  $g\aec_{2}=\aec_{5}$;
    \item if $g=\QQ(\avc_{1},\pi)$, $\Gamma\in[\DD_{3}]$ with $\avc_{1}$ as the primary axis and $\aec_{5}$ as the secondary axis of $\Gamma$;
    \item if $\Gamma\supset \tetra$, then $g$ fixes the three edge axes of the tetrahedron, and $g$ fix the cube. In that case, $\Gamma=\octa$.
  \end{enumerate}
\end{proof}

\subsection{Icosahedral subgroup}

The group $\ico$ is the (orientation-preserving) symmetry group of the dodecahedron $\mathcal{D}_{0}$ (\autoref{fig:dode}), where we have
\begin{itemize}
  \item Twelve vertices: $(\pm \phi,\pm \phi^{-1},0),(\pm \phi^{-1},0,\pm \phi),(0,\pm \phi,\pm \phi^{-1})$, $\phi$ being the gold number.
  \item Eight vertices: $(\pm 1,\pm 1,\pm 1)$ of a cube.
\end{itemize}

\begin{figure}[h]
  \centering
  \includegraphics[scale=1.5]{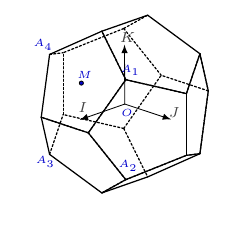}
  \caption{Dodecahedron $\mathcal{D}_{0}$}
  \label{fig:dode}
\end{figure}

We thus have the decomposition
\begin{equation}\label{eq:dico}
  \ico=\biguplus_{i=1}^6\ZZ_{5}^{\afd_{i}}\biguplus_{j=1}^{10}\ZZ_{3}^{\avd_{j}}\biguplus_{l=1}^{15}\ZZ_{2}^{\aed_{l}}
\end{equation}
with \emph{vertices}, \emph{edges} and \emph{faces axes} respectively denoted $\avd_{i}$, $\aed_{j}$ and $\afd_{j}$. For instance we have
\begin{equation*}
  \avd_{1}:=(OA_{1}),\quad \aed_{1}:=(OI),\quad \afd_{1}:=(OM)
\end{equation*}
where $M$ is the center of some face.

From decomposition \eqref{eq:dico} and from Lemma~\ref{lem:intercy} we obtain the following corollary.

\begin{cor}
  Let $n\geq 2$ be some integer. Set
  \begin{equation*}
    d_{2}:=\gcd(n,2)\:; \: d_{3}:=\gcd(n,3)\:; \: d_{5}:=\gcd(n,5).
  \end{equation*}
  Then, we have
  \begin{equation*}
    [\ZZ_{n}]\circledcirc [\ico]=\set{[\triv],[\ZZ_{d_{2}}],[\ZZ_{d_{3}}],[\ZZ_{d_{5}}]}.
  \end{equation*}
\end{cor}

Using once again decomposition~\eqref{eq:dico} and clips operation $[\DD_{n}]\circledcirc[\ZZ_{m}]$ in Lemma~\ref{lem:Clis_Diedre_Cyclic} we get the following corollary.

\begin{cor}
  Let $n\geq 2$ be some integer. Set
  \begin{equation*}
    d_{2}:=\gcd(n,2)\:; \: d_{3}:=\gcd(n,3)\:; \: d_{5}:=\gcd(n,5).
  \end{equation*}
  Then, we have
  \begin{equation*}
    [\DD_{n}]\circledcirc [\ico] = \set{[\triv],[\ZZ_{2}],[\ZZ_{d_{3}}],[\ZZ_{d_{5}}],[\DD_{d_{2}}],[\DD_{d_{3}}],[\DD_{d_{5}}]}.
  \end{equation*}
\end{cor}

\begin{lem}\label{lem:Ico_Tetra}
  We have
  \begin{equation*}
    [\ico]\circledcirc [\tetra]=\set{[\triv],[\ZZ_{2}],[\ZZ_{3}],[\tetra]}.
  \end{equation*}
\end{lem}

\begin{proof}
  Let $\Gamma=\ico\cap \left(g\tetra g^{-1}\right)$. From decompositions~\eqref{eq:dico}--\eqref{dtetra} and Lemma~\ref{lem:intercy}, we only have to consider intersections
  \begin{equation*}
    \ZZ_{3}^{\avd_{j}}\cap \left(g\ZZ_{3}^{\avt_{i}}g^{-1}\right),\quad \ZZ_{2}^{\aed_{l}}\cap \left(g\ZZ_{2}^{\aet_{j}}g^{-1}\right).
  \end{equation*}
  First, note that there always exists $g$ such that $\Gamma$ reduces to a subgroup conjugate to $\triv,\ZZ_{2}$ or $\ZZ_{3}$. Now, if $\Gamma$ contains a subgroup conjugate to $\DD_{2}$, then its three characteristic axes are edge axes of the dodecahedron: say $Ox$, $Oy$ and $Oz$. In that case, $g$ fixes these three axes, and also the 8 vertices of the cube $\mathcal{C}_{0}$. The subgroup $g\tetra g^{-1}$ is thus the (orientation-preserving) symmetry group of a tetrahedron included in the dodecahedron $\mathcal{D}_{0}$. Then,
  $\Gamma\in [\tetra]$.
\end{proof}

The next two Lemmas are more involving.

\begin{lem}
  We have
  \begin{equation*}
    [\octa]\circledcirc [\ico]=\set{[\triv],[\ZZ_{2}],[\ZZ_{3}],[\DD_{3}],[\tetra] }.
  \end{equation*}
\end{lem}

\begin{proof}
  Let $\Gamma=\ico \cap \left(g\octa g^{-1}\right)$. From the poset in Figure~\ref{fig:SO3-lattice}, we deduce that the conjugacy class $[\Gamma]$ belongs to the following list
  \begin{equation*}
    \set{[\triv],[\ZZ_{2}],[\DD_{2}],[\ZZ_{3}],[\DD_{3}],[\tetra]}.
  \end{equation*}
  First, we can always find $g\in \SO(3)$ such that $\Gamma\in[\ZZ_{3}]$, $\Gamma\in [\ZZ_{2}]$ or $\Gamma=\triv$. Moreover, as in the proof of Lemma~\ref{lem:Ico_Tetra}, if $\Gamma$ contains a subgroup conjugate to $\DD_{2}$, then $\Gamma\in [\tetra]$. Finally, we will exhibit some $g\in \SO(3)$ such that $\Gamma\in [\DD_{3}]$. First, recall that
  \begin{equation*}
    A_{1}(1,1,1), \quad A_{2}(1,1,-1), \quad A_4(1,-1,1), \quad A_{5}(-1,-1,1)
  \end{equation*}
  are common vertices of the cube $\mathcal{C}_{0}$ and the dodecahedron $\mathcal{D}_{0}$. Let now $B_{2}(\phi^{-1},0,-\phi)$ be a vertex of the dodecahedron and $I_{2}$ (resp. $I_4$) be the middle-point of $[B_{2}A_{2}]$ (resp. $[A_{4}A_{5}]$ -- see \autoref{fig:dodeclipscube}).

  \begin{figure}[h]
    \centering
    \includegraphics[scale=1.5]{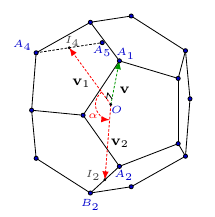}
    \caption{Rotation $g$ to obtain $[\DD_{3}]$ in $[\octa]\circledcirc [\ico]$.}
    \label{fig:dodeclipscube}
  \end{figure}

  Then, $a_{1}=(OI_4)$ and $a_{2}=(OI_{2})$ are perpendicular axes to $a=(OA_{1})$. Choose $\alpha$ such that $\QQ(\vv,\alpha)\vv_{1}=\vv_{2}$, with (see~\autoref{fig:dodeclipscube}):
  \begin{equation*}
    \vv=\overrightarrow{OA_{1}}, \quad \vv_{1}=\overrightarrow{OI_4}, \quad \vv_{2}=\overrightarrow{OI_{2}},
  \end{equation*}
  and set $g=\QQ(\vv,\alpha)$. From decompositions \eqref{eq:deccube} and \eqref{eq:dico}, we deduce then, that $\ico \cap \left(g\octa g^{-1}\right)$ contains the subgroups
  \begin{equation*}
    \ZZ_{2}^{a_{2}}\cap \left(g\ZZ_{2}^{a_{1}}g^{-1}\right) = \ZZ_{2}^{a_{2}},\quad \ZZ_{3}^{a}\cap, \left(g\ZZ_{3}^{a}g^{-1}\right) = \ZZ_{3}^{a}.
  \end{equation*}
  Therefore, $\Gamma$ contains a subgroup conjugate to $\DD_{3}$. Using a maximality argument (see poset in~\autoref{fig:SO3-lattice}), we must have $\Gamma\in[\DD_{3}]$, and this concludes the proof.
\end{proof}

\begin{lem}
  We have
  \begin{equation*}
    [\ico]\circledcirc [\ico] = \set{[\triv],[\ZZ_{2}],[\ZZ_{3}],[\DD_{3}],[\ZZ_{5}],[\DD_{5}],[\ico]}.
  \end{equation*}
\end{lem}

\begin{proof}
  Let $\Gamma=\ico \cap \left(g\ico g^{-1}\right)$. Considering the subclasses of $[\ico]$, we have to check the classes
  \begin{equation*}
    [\tetra], \quad [\DD_{3}], \quad [\DD_{5}], \quad [\DD_{2}], \quad [\ZZ_{3}], \quad [\ZZ_{5}], \quad [\ZZ_{2}].
  \end{equation*}
  Note first, that there exist rotations $g$ such that $\Gamma\in[\ZZ_{2}]$, $\Gamma\in[\ZZ_{3}]$, $\Gamma\in[\ZZ_{5}]$ or $\Gamma=\triv$.

  When $\Gamma$ contains a subgroup conjugate to $\tetra$ or $\DD_{2}$, using the same argument as in the proof of Lemma~\ref{lem:Ico_Tetra}, $g$ fixes all the dodecahedron vertices. In that case, $\Gamma=\ico$.

  We will now exhibit some $g\in \SO(3)$ such that $\Gamma\in [\DD_{3}]$. Consider the dodecahedron $\mathcal{D}_{0}$ in \autoref{fig:dode} and the points $A_{3}(1,-1,-1)$ and $B_{3}(\phi,-\phi^{-1},0)$. Let $I_{3}$ be the middle-point of $[A_{3}B_{3}]$ and $g$ be the order two rotation around $a_{1}:=(OA_{1})$ (see~\autoref{fig:dodeclipsdode}). Let
  \begin{equation*}
    b_{1}:=(OI_{3}),\quad b_{2}:=\QQ\left(a_{1},\frac{2\pi}{3}\right)b_{1},\quad b_{3}:=\QQ\left(a_{1},\frac{2\pi}{3}\right)b_{2}.
  \end{equation*}
  We check directly that $a_{1},b_{i}$ ($i=1,\dotsc,3)$ are the only $g$-invariant characteristic axes of the dodecahedron. We deduce then, from decomposition~\eqref{eq:dico} that $\Gamma$ reduces to
  \begin{equation*}
    \ZZ_{3}^{a_{1}}\biguplus_{i=1}^{3} \ZZ_{2}^{b_{i}}\in [\DD_{3}].
  \end{equation*}

  \begin{figure}[h]
    \centering
    \includegraphics[scale=1.5]{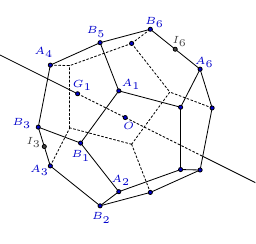}
    \caption{Rotation $g$ to obtain $[\DD_{3}]$ or $[\DD_{5}]$ in $[\ico]\circledcirc [\ico]$.}
    \label{fig:dodeclipsdode}
  \end{figure}

  In the same way, we can find $g\in\SO(3)$, such that $\Gamma\in [\DD_{5}]$. Let
  \begin{equation*}
    B_{1}(\phi,\phi^{-1},0), \quad B_{5}(\phi^{-1},0,\phi), \quad B_{6}(-\phi^{-1},0,\phi), \quad A_{6}(-1,1,1)
  \end{equation*}
  be vertices of the dodecahedron $\mathcal{D}_{0}$. Let $G_{1}$ be the center of the pentagon $A_{1}B_{1}B_{3}A_{4}B_{5}$ (see~\autoref{fig:dodeclipsdode}) and $I_{6}$ be the middle-point of $[B_{6}A_{6}]$. Let $g$ be the order two rotation around $f_{1}:=(OG_{1})$ and set
  \begin{equation*}
    c_{1}:=(OI_{6}),\quad c_{k+1}:=\QQ\left(\overrightarrow{OA_{1}},\frac{2\pi}{5}\right)c_k,\quad 1\le k\le 4.
  \end{equation*}
  Then we can check that $f_{1}$, $c_k$ ($k=1,\dotsc,5$) are the only $g$-invariant characteristic axes of the dodecahedron. Using decomposition~\eqref{eq:dico}, we deduce then, that $\Gamma\in [\DD_{5}]$, which concludes the proof.
\end{proof}

\subsection{Infinite subgroups}

The primary axis of both $\SO(2)$ and $\OO(2)$ is defined as the $z$-axis, while any perpendicular axis to $(Oz)$ is a secondary axis for $\OO(2)$.

Clips operation between $\SO(2)$ or $\OO(2)$ and finite subgroups are obtained using simple arguments on characteristic axes. The same holds for the clips $[\SO(2)]\circledcirc [\SO(2)]$. To compute $[\OO(2)]\circledcirc [\OO(2)]$, consider the subgroup $\Gamma=\OO(2)\cap \left(g\OO(2) g^{-1}\right)$ for some $g\in \SO(3)$.
\begin{enumerate}
  \item If both primary axes are the same, then $\Gamma=\OO(2)$;
  \item if the primary axis of $g\OO(2) g^{-1}$ is in the $xy$--plane, then $\Gamma\in[\DD_{2}]$;
  \item in all other cases, $\Gamma\in[\ZZ_{2}]$, where the primary axis of $\Gamma$ is perpendicular to the primary axes of $\OO(2)$ and $g\OO(2) g^{-1}$.
\end{enumerate}

\section{Proofs for $\OO(3)$}
\label{sec:proofs-O3}

In this Appendix, we provide the details about clips operations between type III closed $\OO(3)$ subgroups. The proofs follow the same ideas that has been used for $\SO(3)$ closed subgroups (decomposition into simpler subgroups and discussion about their characteristic axes), but most of them are unfortunately more involving.

We first recall the general structure of type III subgroups $\Gamma$ of $\OO(3)$ (see~\cite{IG1984} for details). For each such subgroup $\Gamma$, there exists a couple $L\subset H$ of $\SO(3)$ subgroups such that $H=\pi(\Gamma)$, where
\begin{equation*}
  \pi\: : \: g\in \OO(3)\mapsto \det(g) g\in \SO(3)
\end{equation*}
and $L=\SO(3)\cap \Gamma$ is an indexed $2$ subgroup of $H$. These characteristic couples are detailed in~\autoref{tab:ClassIII}. Note that, for a given couple $(L,H)$, $\Gamma$ can be recovered as $\Gamma = L\cap g L$, where $-g\in H\setminus L$.

\begin{table}[h]
  \begin{center}
    \begin{tabular}{ccc}
      \toprule
      \hspace*{1cm}$\Gamma$\hspace*{1cm} & \hspace*{1cm}$H$\hspace*{1cm} & \hspace*{1cm}$L$\hspace*{1cm} \\
      \midrule
      $\ZZ_{2}^{-}$                      & $\ZZ_{2}$                     & $\triv$                       \\

      $\ZZ_{2n}^{-}$                     & $\ZZ_{2n}$                    & $\ZZ_{n}$                     \\

      $\DD_{n}^{v}$                      & $\DD_{n}$                     & $\ZZ_{n}$                     \\

      $\DD_{2n}^h$                       & $\DD_{2n}$                    & $\DD_{n}$                     \\

      $\octa^{-}$                        & $\octa$                       & $\tetra$                      \\

      $\OO(2)^{-}$                       & $\OO(2)$                      & $\SO(2)$                      \\
      \bottomrule
    \end{tabular}
    \caption{Characteristic couples for type III subgroups}
    \label{tab:ClassIII}
  \end{center}
\end{table}

In the following, we shall use the following convention:
\begin{equation*}
  \ZZ_{1}^{\sigma} = \ZZ_{1}^{-} = \DD_{1}^{v} = \triv.
\end{equation*}

\subsection{$\ZZ_{2n}^{-}$ subgroups}

Consider the couple $\ZZ_{n} \subset \ZZ_{2n}$ ($n>1$) in~\autoref{tab:ClassIII}, where
\begin{equation*}
  \ZZ_{2n} = \set{\Idd,\QQ\left(\kk;\frac{\pi}{n}\right),\QQ\left(\kk;\frac{2\pi}{n}\right),\dotsc}
\end{equation*}
and let $\mathbf{r}_{n} := \displaystyle{\QQ\left(\kk;\frac{\pi}{n}\right)} \in \ZZ_{2n}\setminus \ZZ_{n}$. Set
\begin{equation}\label{eq:Z2nMoins}
  \ZZ_{2n}^{-}:=\ZZ_{n} \cup (-\mathbf{r}_{n}\ZZ_{n}).
\end{equation}
The \emph{primary axis} of the subgroup $\ZZ_{2n}^{-}$ is defined as the $z$-axis.

\begin{rem}\label{rem:Z2moins}
  The subgroup $\ZZ_{2}^{-}$ is generated by $-\QQ(\kk,\pi)$ which is the reflection through the $xy$ plane. If $\sigma_{b}$ is the reflection through the plane with normal axis $b$, then $\ZZ_{2}^{\sigma_b} := \set{e,\sigma_{b}}$, which is conjugate to $\ZZ_{2}^{-}$.
\end{rem}

We have the following lemma.

\begin{lem}\label{lem:InterMoinsRn}
  Let $m,n\geq 2$ be two integers. Set $d:=\gcd(n,m)$ and
  \begin{equation*}
    \mathbf{r}_{n} := \QQ\left(\kk;\frac{\pi}{n}\right), \quad \mathbf{r}_{m} := \QQ\left(\kk;\frac{\pi}{m}\right).
  \end{equation*}
  The intersection $(-\mathbf{r}_{n}\ZZ_{n})\cap (-\mathbf{r}_{m}\ZZ_{m})$ does not reduce to $\emptyset$ if and only if $m/d$ and $n/d$ are odds. In such a case, we have
  \begin{equation*}
    (-\mathbf{r}_{n}\ZZ_{n})\cap (-\mathbf{r}_{m}\ZZ_{m}) = -\mathbf{r}_{d}\ZZ_{d},\quad \mathbf{r}_{d} = \QQ\left(\kk;\frac{\pi}{d}\right).
  \end{equation*}
\end{lem}

\begin{proof}
  The intersection $(-\mathbf{r}_{n}\ZZ_{n})\cap (-\mathbf{r}_{m}\ZZ_{m})$ differs from $\emptyset$, if and only if, there exist integers $i,j$ such that
  \begin{equation*}
    \frac{2i+1}{n}\pi = \frac{2j+1}{m}\pi,\quad 2i+1\leq 2n,\quad 2j+1\leq 2m.
  \end{equation*}
  Let $n=dn_{1}$ and $m=dm_{1}$. The preceding equation can then be recast as $(2i+1)m_{1}=(2j+1)n_{1}$, so that
  \begin{equation*}
    2i+1=pn_{1} \text{ and } 2j+1=pm_{1}.
  \end{equation*}
  Thus, $m_{1}$ and $n_{1}$ are necessarily odds, in which case (recall that $\ZZ_{1} = \triv$)
  \begin{equation*}
    (-\mathbf{r}_{n}\ZZ_{n})\cap (-\mathbf{r}_{n}\ZZ_{m}) = -\mathbf{r}_{d}\ZZ_{d},\quad \mathbf{r}_{d} = \QQ\left(\kk;\frac{\pi}{d}\right).
  \end{equation*}
\end{proof}

\begin{cor}\label{lem:ClipsZmoins}
  Let $m,n\geq 1$ be two integers. Set $d:=\gcd(n,m)$. Then, we have
  \begin{equation*}
    \left[ \ZZ_{2n}^{-} \right] \circledcirc \left[ \ZZ_{2m}^{-}\right]
    =\begin{cases}
    \set{[\triv],[\ZZ_{2d}^-]} \text{if $n/d$ and $m/d$ are odd} \\
    \set{[\triv],[\ZZ_{d}]} \text{otherwise}
    \end{cases}
  \end{equation*}
\end{cor}

\begin{proof}
  Note first that all intersections reduce to $\triv$ when the characteristic axes don't match, so we have only to consider the situation where they match. Now, by~\eqref{eq:Z2nMoins}, we have only to consider the intersection
  \begin{equation*}
    \ZZ_{2n}^{-}\cap \ZZ_{2m}^{-} =(\ZZ_{n}\cap \ZZ_{m})\cup \left((-\mathbf{r}_{n}\ZZ_{n})\cap (-\mathbf{r}_{m}\ZZ_{m})\right).
  \end{equation*}
  By Lemma~\ref{lem:intercy}, $\ZZ_{n}\cap \ZZ_{m}=\ZZ_{d}$ and we directly conclude using Lemma~\ref{lem:InterMoinsRn}.
\end{proof}

\subsection{$\DD_{n}^{v}$ subgroups}

Consider the couple $\ZZ_{n} \subset \DD_{n}$ in~\autoref{tab:ClassIII}. Recall that $\DD_{n}$ contains $\ZZ_{n}$ and all the second order rotations about the $b_{j}$'s axes (see~\eqref{diedre} and \autoref{fig:dihedral-second-axis}). Set
\begin{equation}\label{eq:diedralv}
  \DD_{n}^{v} := \ZZ_{n} \biguplus_{j=1}^{n} \ZZ_{2}^{\sigma_{b_{j}}}.
\end{equation}
Given $g\in \OO(3)$, the primary axis of $g\DD_{n}^{v}g^{-1}$ is $g(Oz)$, and its secondary axes are $gb_{j}$.

\begin{lem}\label{lem:DnvClipsZ2mmoins}
  Let $n\geq 2$, $m\geq 1$ be two integers. Set $d=\gcd(n,m)$ and
  \begin{equation*}
    i(m):=3-\gcd(2,m)=
    \begin{cases}
      1, & \text{ if $m$ is even}, \\
      2, & \text{ if $m$ is odd}.
    \end{cases}
  \end{equation*}
  Then, we have
  \begin{equation*}
    \left[ \DD_{n}^{v} \right] \circledcirc \left[ \ZZ_{2m}^{-}\right] = \set{\triv,\left[ \ZZ_{i(m)}^{-}\right],\left[ \ZZ_{d} \right]}.
  \end{equation*}
\end{lem}

\begin{proof}
  Let $\Gamma := \DD_{n}^{v} \cap \left(g\ZZ_{2m}^{-} g^{-1}\right)$ and
  \begin{equation*}
    \ZZ_{2m}^{-}=\ZZ_{m} \cup (-\mathbf{r}_{m}\ZZ_{m}),\quad \mathbf{r}_{m}=\QQ\left(\kk;\frac{\pi}{m}\right).
  \end{equation*}
  \begin{enumerate}
    \item If both primary axes of $\DD_{n}^{v}$ and $g\ZZ_{2m}^{-} g^{-1}$ (generated by $g\kk$) match, then by decomposition~\eqref{eq:diedralv} and Lemma~\ref{lem:intercy}, $\Gamma$ reduces to $\ZZ_{n}\cap \ZZ_{m}=\ZZ_{d}$.
    \item If the primary axis of $g\ZZ_{2m}^{-}g^{-1}$ matches with a secondary axis of $\DD_{n}^{v}$, say $(Ox)$, then $\Gamma$ reduces to $\ZZ_{2}^{\sigma_{b_{0}}}\cap \left(g\ZZ_{2m}^{-}g^{-1}\right)$. Such an intersection has already be studied in the clips operation $[\ZZ_{2}^{-}] \circledcirc [\ZZ_{2m}^{-}]$  (see Lemma~\ref{lem:ClipsZmoins}).
    \item Otherwise, $\Gamma=\triv$, which concludes the proof.
  \end{enumerate}
\end{proof}

\begin{lem}
  Let $m,n\geq 2$ be two integers and $d=\gcd(n,m)$. Then, we have
  \begin{equation*}
    \left[ \DD_{n}^{v} \right] \circledcirc \left[ \DD_{m}^{v}\right]=\set{\triv,\left[ \ZZ_{2}^{-}\right],\left[ \DD_{d}^{v} \right],\left[ \ZZ_{d} \right]}.
  \end{equation*}
\end{lem}

\begin{proof}
  Only two cases need to be considered.
  \begin{enumerate}
    \item If the primary axes of $\DD_{n}^{v}$ and $g\DD_{m}^{v}g^{-1}$ do not match, then we get $\triv$.
    \item If they have the same primary axis, by decomposition~\eqref{eq:diedralv}, we have to consider the intersections
          \begin{equation*}
            \ZZ_{n}\cap \ZZ_{m},\quad \ZZ_{2}^{\sigma_{b_{j}}}\cap \ZZ_{2}^{\sigma_{b'_k}},
          \end{equation*}
          which reduce to
          \begin{equation*}
            \ZZ_{d}\biguplus \ZZ_{2}^{\sigma_{c_{l}}},
          \end{equation*}
          where $c_{l}$ are the common secondary axis of the two subgroups. Then, we get either $\ZZ_{d}$, $\DD_{d}^{v}$ or a subgroup conjugate to $\ZZ_{2}^{-}$ (when $d=1$ and $b_{0}=b'_{0}$), which concludes the proof.
  \end{enumerate}
\end{proof}

\subsection{$\DD_{2n}^h$ subgroups}

Consider the couple $\DD_{n} \subset \DD_{2n}$ in~\autoref{tab:ClassIII}. For $j=0,\dotsc , n-1$, let $p_{j}$ be the axis generated by
\begin{equation*}
  \vv_{j} := \QQ\left(\kk;\frac{j\pi}{n}\right)\cdot \ii,
\end{equation*}
and $q_{j}$, be the axis generated by
\begin{equation*}
  \ww_{j}:=\QQ\left(\kk;\frac{(2j+1)\pi}{2n}\right)\cdot \ii.
\end{equation*}

\begin{figure}[h]
  \centering
  \subfigure[Characteristic axes for $\DD_{6}\supset \DD_{3}$\label{fig:D3-D6}]{\includegraphics[scale=0.9]{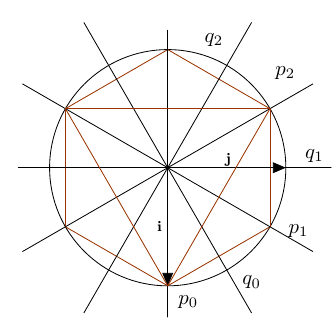}}
  \subfigure[Characteristic axes for $\DD_{8}\supset \DD_{4}$\label{fig:D4-D8}]{\includegraphics[scale=0.9]{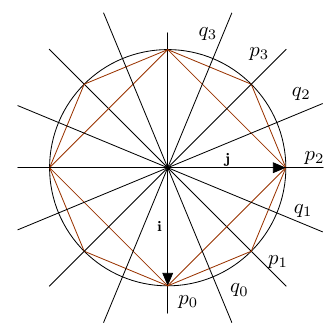}}
\end{figure}

Set
\begin{equation*}
  \DD_{n}=\set{ \triv,\QQ\left(\kk;\frac{2\pi}{n}\right),\QQ\left(\kk;\frac{4\pi}{n}\right), \dotsc, \QQ\left(\vv_{0};\pi\right),\QQ\left(\vv_{1};\pi\right),\dotsc },
\end{equation*}
and
\begin{equation*}
  -\mathbf{r}_{n}\DD_{n}=\set{ -\QQ\left(\kk;\frac{\pi}{n}\right),-\QQ\left(\kk;\frac{3\pi}{n}\right), \dotsc, -\QQ\left(\ww_{0};\pi\right),-\QQ\left(\ww_{1};\pi\right),\dotsc },
\end{equation*}
where $\mathbf{r}_{n}=\displaystyle{\QQ\left(\kk;\frac{\pi}{n}\right)}$. We define
\begin{equation}\label{eq:D2nh}
  \DD_{2n}^h := \DD_{n} \cup \left(-\mathbf{r}_{n}\DD_{n} \right),
\end{equation}
which decomposes as
\begin{equation}\label{eq:diedralh}
  \DD_{2n}^h=\ZZ_{2n}^{-}\biguplus_{j=0}^{n-1} \ZZ_{2}^{p_{j}} \biguplus_{j=0}^{n-1} \ZZ_{2}^{\sigma_{q_{j}}} .
\end{equation}
Note that in this decomposition, there are $n$ subgroups conjugate to $\ZZ_{2}$ and $n$ others conjugate to $\ZZ_{2}^{-}$. The $z$-axis (resp $x$-axis) is said to be the primary (resp. secondary) axis of $\DD_{2n}^h$. For each $g\in \OO(3)$, the primary (resp. secondary) axis of the subgroup $g\DD_{2n}^{h}g^{-1}$ is generated by $g\kk$ (resp. by $g\ii$).

\begin{lem}\label{lem:D2nhClipsZ2mmoins}
  Let $m,n\geq 2$ be two integers. Set $d=\gcd(n,m)$, $d_{2}(m)=\gcd(m,2)$ and
  \begin{equation*}
    i(m)=
    \begin{cases}
      1, & \text{if $m$ is even}, \\
      2, & \text{otherwise}.
    \end{cases}
  \end{equation*}
  Then,
  \begin{itemize}
    \item If $\cfrac{n}{d}$ or $\cfrac{m}{d}$ is even, we have
          \begin{equation*}
            [\DD_{2n}^h]\circledcirc [\ZZ_{2m}^{-}]=\set{\triv,[\ZZ_{d_{2}(m)}],[\ZZ_{i(m)}^{-}],[\ZZ_{d}]};
          \end{equation*}
    \item If $\cfrac{n}{d}$ and $\cfrac{m}{d}$ are odd, we have
          \begin{equation*}
            [\DD_{2n}^h]\circledcirc [\ZZ_{2m}^{-}]=\set{\triv,[\ZZ_{d_{2}(m)}],[\ZZ_{i(m)}^{-}],[\ZZ_{2d}^{-}]}.
          \end{equation*}
  \end{itemize}
\end{lem}

\begin{proof}
  First of all, if no characteristic axes of $\DD_{2n}^h$ and $g\ZZ_{2m}^{-}g^{-1}$ match, then their intersection reduces to $\triv$.

  We have now to consider three cases:

  \noindent $(1)$ The first case is when $\DD_{2n}^h$ and $g\ZZ_{2m}^{-}g^{-1}$ have the same primary axis. Then, using decompositions \eqref{eq:diedralh} and \eqref{eq:Z2nMoins}, we only have to consider the intersection
  \begin{equation*}
    \ZZ_{2n}^{-}\cap \ZZ_{2m}^{-}.
  \end{equation*}
  This has already been studied in the clips operation $[\ZZ_{2n}^{-}]\circledcirc [\ZZ_{2m}^{-}]$ in Lemma~\ref{lem:ClipsZmoins}, leading to the conjugacy class $[\ZZ_{d}]$ or $[\ZZ_{2d}^{-}]$.

  \noindent $(2)$ The second one is when some secondary axis $p_{j}$ (say $p_{0}$) match with the primary axis of $g\ZZ_{2m}^{-}g^{-1}$, then we only have to consider intersection
  \begin{equation*}
    \ZZ_{2}^{p_{0}}\cap \left(g\ZZ_{2m}^{-}g^{-1}\right)=\ZZ_{2}^{p_{0}}\cap \left(g\ZZ_{m}g^{-1}\right)
  \end{equation*}
  leading to $\ZZ_{2}^{p_{0}}$ if $m$ is even.

  \noindent $(2)$  Finally, we have to consider the case when the primary axis of $g\ZZ_{2m}^{-}g^{-1}$ is $q_{0}$. In that case the problem reduces to the intersection
  \begin{equation*}
    \ZZ_{2}^{\sigma_{q_{0}}}\cap \left(g\ZZ_{2m}^{-}g^{-1}\right)
  \end{equation*}
  leading to the conjugacy class $[\ZZ_{2}^{-}]$ for odd $m$ (see Lemma~\ref{lem:ClipsZmoins}). This concludes the proof.
\end{proof}

The cases $[\DD_{2n}^{h}]\circledcirc [\DD_{n}^{v}]$ and $[\DD_{2n}^{h}]\circledcirc [\DD_{2m}^h]$ are more involving. We start by formulating the following lemma, without proof (see \autoref{fig:dihedral-second-axis} for an example):

\begin{lem}\label{lem:Perp_Axe_Diedh}
  If $n$ is even then there exist $p_k\perp p_{l}$ and $q_r\perp q_s$, where no axes $p_{i},q_{j}$ are perpendicular. If $n$ is odd, there exist $p_{i}\perp q_{j}$ and no axes $p_k,p_{l}$, nor $q_r,q_s$ are perpendicular.
\end{lem}

\begin{lem}\label{lem:Dmh_Et_Dmv}
  Let $m,n\geq 2$ be two integers. Set $d_{2}(m):=\gcd(m,2)$,
  \begin{equation*}
    i(m,n) :=
    \begin{cases}
      2, & \text{if $m$ is even and $n$ is odd}, \\
      1, & \text{otherwise},
    \end{cases}
  \end{equation*}
  and
  \begin{equation*}
    i(m) :=3-\gcd(2,m)=
    \begin{cases}
      1, & \text{if $m$ is even}, \\
      2, & \text{if $m$ is odd}.
    \end{cases}
  \end{equation*}
  Then, we have
  \begin{equation*}
    [\DD_{2n}^h]\circledcirc [\DD_{m}^{v}]=\set{[\triv],\left[\ZZ_{i(m)}^{-}\right],\left[ \ZZ_{d_{2}(m)}\right],\left[ \DD_{i(m,n)}^{v}\right],\left[ \ZZ_{d}\right],\left[\DD_{d}^{v}\right]}.
  \end{equation*}
\end{lem}

\begin{proof}
  The only non trivial cases are when $g\DD_{m}^{v}g^{-1}$ and $\DD_{2n}^h$ have no matching characteristic axes. Now we have to distinguish wether the principal axis $a$ of $g\DD_{m}^{v}g^{-1}$ is $(Oz)$ or not:

  \begin{enumerate}[(1)]
    \item Let first suppose that $a=(Oz)$. In that case, we need to compute the intersection $\DD_{2n}^h \cap \DD_{m}^{v}$. From~\eqref{eq:diedralv} and \eqref{eq:diedralh}, this reduces to study the following three intersections
          \begin{equation*}
            \ZZ_{2n}^{-}\cap \ZZ_{m},\quad \ZZ_{2n}^{-}\cap \ZZ_{2}^{\sigma_{b_{j}}},\quad \ZZ_{2}^{\sigma_{q_k}}\cap \ZZ_{2}^{\sigma_{b_{j}}}.
          \end{equation*}
          Now:
          \begin{enumerate}[(a)]
            \item The first intersection, $\ZZ_{2n}^{-}\cap \ZZ_{m}$, reduces to $\ZZ_{n}\cap \ZZ_{m}=\ZZ_{d}$ (from~\eqref{eq:Z2nMoins} and Lemma~\ref{lem:intercy}).
            \item The second one, $\ZZ_{2n}^{-}\cap \ZZ_{2}^{\sigma_{b_{j}}}$, reduces to $\triv$, since primary axes of $\ZZ_{2n}^{-}$ and $\ZZ_{2}^{\sigma_{b_{j}}}$ (conjugate to $\ZZ_{2}^{-}$) do not match.
            \item The last one, $\ZZ_{2}^{\sigma_{q_k}}\cap \ZZ_{2}^{\sigma_{b_{j}}}$, can reduce to some $\ZZ_{2}^{\sigma_{q_k}}$ if $b_{0}=q_{0}$. In that case, $\DD_{2n}^h \cap \DD_{m}^{v}$ contains $\ZZ_{d}$ and some $\ZZ_{2}^{\sigma_{q_k}}$, which generate $\DD_{d}^{v}$ (see~\eqref{eq:diedralv}).
          \end{enumerate}
          This first case thus leads to $\ZZ_d$ or $\DD_d^v$.
    \item Consider now the case when $a\neq (Oz)$. Thus the intersections to be considered are
          \begin{equation*}
            \ZZ_{2n}^{-}\cap \ZZ_{2}^{\sigma_{b'_{0}}},\quad \ZZ_{2}^{p_{j}}\cap \ZZ_{m}^a,\quad \ZZ_{2}^{\sigma_{q_k}}\cap \ZZ_{2}^{\sigma_{b'_{j}}}.
          \end{equation*}
          where the $b'_j$ axes are the secondary axis of $g\DD_{m}^{v}g^{-1}$. Now:
          \begin{enumerate}[(a)]
            \item First suppose that $a=p_0$ (for instance) and all the other axis are different. Then, $\ZZ_{2}^{p_{j}}\cap \ZZ_{m}^a=\ZZ_{d_2(m)}^a$.
            \item Suppose now that $a=p_0$ and $b'j=q_k$ for some couple $(k,j)$, so that $\ZZ_{2}^{\sigma_{q_k}}\cap \ZZ_{2}^{\sigma_{b'_{j}}}=\ZZ_{2}^{\sigma_{q_k}}$. As $q_k\perp p_0$, we deduce from Lemma~\ref{lem:Perp_Axe_Diedh} that $n$ is odd. All depend on $m$ parity: if $m$ is even then $\Gamma$ contains $\ZZ_{2}^{p_0}$ and $\ZZ_{2}^{\sigma_{q_k}}$, and we obtain some subgroup conjugate to $\DD_{2}^{v}$. If $m$ is odd, then $\Gamma$ reduces to $\ZZ_{2}^{\sigma_{q_k}}$, which is conjugate to $\ZZ_{2}^-$.
            \item Finally, suppose that $a\neq p_j$ for all $j$ (and recall that $a\neq (Oz)$), so that the intersections to be considered are
                  \begin{equation*}
                    \ZZ_{2n}^{-}\cap \ZZ_{2}^{\sigma_{b'_{0}}},\quad \ZZ_{2}^{\sigma_{q_k}}\cap \ZZ_{2}^{\sigma_{b'_{j}}}.
                  \end{equation*}	
                  We thus only obtain some subgroups already considered in the previous cases, which conclude the proof.
          \end{enumerate}
  \end{enumerate}
\end{proof}

\begin{lem}\label{lem:deuxdiedreh}
  Let $m,n\geq 2$ be two integers. Set $d=\gcd(n,m)$ and
  \begin{equation*}
    \Delta=[\DD_{2n}^h]\circledcirc [\DD_{2m}^h].
  \end{equation*}
  Then, $[\triv]\subset \Delta$ and:
  \begin{itemize}
    \item For every integer $d$:
          \begin{itemize}
            \item[$\diamond$] If $m$ and $n$ are even, then $\Delta \supset \set{[\ZZ_{2}],[\DD_{2}]}$;
            \item[$\diamond$] If $m$ and $n$ are odds, then $\Delta \supset \set{[\ZZ_{2}^{-}]}$;
            \item[$\diamond$] Otherwise, $\Delta \supset \set{[\ZZ_{2}],[\DD_{2}^{v}]}$;
          \end{itemize}
    \item If $d=1$, then
          \begin{itemize}
            \item[$\diamond$] If $m$ and $n$ are odds, then $\Delta \supset \set{[\DD_{2}^{v}]}$;
            \item[$\diamond$] Otherwise $m$ or $n$ is even and $\Delta \supset \set{[\ZZ_{2}],[\ZZ_{2}^{-}]}$;
          \end{itemize}
    \item If $d\neq 1$, then
          \begin{itemize}
            \item[$\diamond$] If $\dfrac{m}{d}$ and $\dfrac{n}{d}$ are odds, then $\Delta \supset \set{[\ZZ_{2d}^{-}],[\DD_{2d}^h]}$;
            \item[$\diamond$] Otherwise, $\dfrac{m}{d}$ or $\dfrac{n}{d}$ is even and $\Delta \supset \set{ [\ZZ_{d}],[\DD_{d}],[\DD_{d}^{v}]}$;
          \end{itemize}
  \end{itemize}
\end{lem}

\begin{proof}[Sketch of proof]
  We consider decomposition~\eqref{eq:D2nh}. If no characteristic axes $\DD_{2n}^{h}$ and $g\DD_{2m}^{h}g^{-1}$ match, then their intersection reduces to $\triv$. Otherwise, from~\eqref{eq:D2nh} it reduces to
  \begin{align*}
    \ZZ_{2n}^{-}\cap \left(g\ZZ_{2m}^{-}g^{-1}\right)             & ,\quad \ZZ_{2n}^{-}\cap \left(g\ZZ_{2}^{\sigma_{q'_k}}g^{-1}\right),\quad \ZZ_{2}^{p_{j}}\cap \left(g\ZZ_{2}^{p'_k}g^{-1}\right), \\
    \ZZ_{2}^{\sigma_{q_{j}}}\cap \left(g\ZZ_{2m}^{-}g^{-1}\right) & ,\quad \ZZ_{2}^{\sigma_{q_{j}}}\cap \left(g\ZZ_{2}^{\sigma_{q'_k}}g^{-1}\right),
  \end{align*}
  where all $\ZZ_{2}^{\sigma_{q_{j}}},\ZZ_{2}^{\sigma_{q'_k}}$ are subgroups conjugate to $\ZZ_{2}^{-}$.

  Now all these intersections have already been studied in the clips operation $[\ZZ_{2r}^{-}] \circledcirc [\ZZ_{2s}^{-}]$. We can thus use Lemma~\ref{lem:ClipsZmoins}, argue on the characteristic axes and Lemma~\ref{lem:Perp_Axe_Diedh}, to conclude the proof in each case.
\end{proof}

\subsection{$\octa^{-}$ subgroup}
\label{subsubsec:OMoins}

Consider the couple $\tetra \subset \octa$ in~\autoref{tab:ClassIII} and the following decompositions
\begin{equation*}
  \octa=\biguplus_{i=1}^3 \ZZ_4^{\afc_{i}} \biguplus_{j=1}^4 \ZZ_{3}^{\avc_{j}} \biguplus_{l=1}^6 \ZZ_{2}^{\aec_{l}},
\end{equation*}
and
\begin{equation*}
  \tetra=\biguplus_{j=1}^4 \ZZ_{3}^{\avt_{j}} \biguplus \ZZ_{2}^{\aet_{1}}\biguplus \ZZ_{2}^{\aet_{2}}\biguplus \ZZ_{2}^{\aet_{3}},\quad \ZZ_{2}^{\aet_{i}}\subset \ZZ_4^{\afc_{i}},\quad i=1,2,3.
\end{equation*}
This leads (see~\cite{IG1984} for details) to the decomposition
\begin{equation}\label{eq:omoins}
  \octa^{-}:=\biguplus_{i=1}^3(\ZZ_4^{\afc_{i}})^{-}\biguplus_{j=1}^4 \ZZ_{3}^{\avc_{j}}\biguplus_{l=1}^6 \ZZ_{2}^{\sigma_{\aec_{l}}},
\end{equation}
where $(\ZZ_4^{\afc_{i}})^{-}$ is the subgroup conjugate to $\ZZ_{4}^{-}$ with $\afc_{i}$ as primary axis. Note also that $\ZZ_{2}^{\sigma_{e_{l}}}$ are subgroups conjugate to $\ZZ_{2}^{-}$ with $\aec_{l}$ as primary axis.

Using this decomposition~\eqref{eq:omoins}, and those of type III closed $\OO(3)$ subgroups previously mentioned directly lead to the following corollaries.

\begin{cor}
  Let $n\geq 2$ be an integer. Set $d_{2}(n)=\gcd(n,2)$ and $d_{3}(n)=\gcd(3,n)$. Then,
  \begin{itemize}
    \item if $n$ is odd, we have
          \begin{equation*}
            [\octa^{-}]\circledcirc [\ZZ_{2n}^{-}]=\set{[\triv],[\ZZ_{2}^{-}],[\ZZ_{d_{3}(n)}]};
          \end{equation*}
    \item if $n=2+4k$ for $k\in \mathbb{N}$, we have
          \begin{equation*}
            [\octa^{-}]\circledcirc [\ZZ_{2n}^{-}]=\set{[\triv],[\ZZ_4^{-}],[\ZZ_{d_{3}(n)}]};
          \end{equation*}
    \item if $n$ is even and $4\nmid n$, we have
          \begin{equation*}
            [\octa^{-}]\circledcirc [\ZZ_{2n}^{-}]=\set{[\triv],[\ZZ_{2}],[\ZZ_{d_{3}(n)}]}.
          \end{equation*}
  \end{itemize}
  Moreover in all cases, we have
  \begin{equation*}
    [\octa^{-}] \circledcirc [\DD_{n}^{v}] = \set{[\triv],[\ZZ_{2}^{-}],[\ZZ_{d_{3}(n)}],[\DD_{d_{3}(n)}^{v}],[\ZZ_{d_{2}(n)}],[\DD_{d_{2}(n)}^{v}]}.
  \end{equation*}
\end{cor}

\begin{cor}
  Let $n\geq 2$ be an integer and $d_{3}(n):=\gcd(n,3)$.
  \begin{itemize}
    \item If $n$ is even and $n=2+4k$ for $k\in \mathbb{N}$, then we have
          \begin{equation*}
            [\octa^{-}]\circledcirc [\DD_{2n}^{h}]=\set{[\triv],[\ZZ_4^{-}],[\DD_4^h],[\ZZ_{d_{3}(n)}],[\DD_{d_{3}(n)}^{v}]};
          \end{equation*}
    \item if $n$ is even and $4\mid n$, then we have
          \begin{equation*}
            [\octa^{-}]\circledcirc [\DD_{2n}^{h}]=\set{[\triv],[\ZZ_{2}],[\DD_{2}],[\DD_{2}^{v}],[\ZZ_{d_{3}(n)}],[\DD_{d_{3}(n)}^{v}]}.
          \end{equation*}
    \item if $n$ is odd, then we have
          \begin{equation*}
            [\octa^{-}]\circledcirc [\DD_{2n}^{h}]=\set{[\triv],[\ZZ_{2}],[\ZZ_{2}^{-}],[\DD_{2}],[\DD_{2}^{v}],[\ZZ_{d_{3}(n)}],[\DD_{d_{3}(n)}^{v}]}.
          \end{equation*}
  \end{itemize}
\end{cor}

\begin{cor}
  We have
  \begin{equation*}
    [\octa^{-}]\circledcirc [\octa^{-}]=\set{[\triv],[\ZZ_{2}^{-}],[\ZZ_4^{-}],[\ZZ_{3}]}.
  \end{equation*}
\end{cor}

\subsection{$\OO(2)^{-}$ subgroup}

Consider the couple $\SO(2) \subset \OO(2)$ in~\autoref{tab:ClassIII} and set
\begin{equation}\label{eq:DecO2moins}
  \OO(2)^{-} := \SO(2)\biguplus_{v\subset xy\text{-plane}} \ZZ_{2}^{\sigma_{v}}.
\end{equation}

As $\ZZ_{2}^{\sigma_{v}}$ are subgroups conjugate to $\ZZ_{2}^{-}$, previous results on clips operation of $[\ZZ_{2}^{-}]$ and Type III subgroups except $[\OO(2)^{-}]$ leads to the following lemma.

\begin{lem}
  Let $n\geq 2$ be some integer. Set $d_{2}(n):=\gcd(2,n)$ and
  \begin{equation*}
    i(n):=3-\gcd(2,n)=
    \begin{cases}
      1, & \text{if $n$ is even}, \\
      2, & \text{if $n$ is odd}.
    \end{cases}
  \end{equation*}
  Then, we have
  \begin{align*}
    [\OO(2)^{-}] \circledcirc [\ZZ_{2n}^{-}] & = \set{[\triv],[\ZZ_{i(n)}^{-}],[\ZZ_{n}]},                                    \\
    [\OO(2)^{-}] \circledcirc [\DD_{n}^{v}]  & = \set{[\triv],[\ZZ_{2}^{-}],[\DD_{n}^{v}]},                                   \\
    [\OO(2)^{-}] \circledcirc [\DD_{2n}^h]   & = \set{[\triv],[\ZZ_{d_{2}(n)}],[\ZZ_{2}^{-}],[\DD_{i(n)}^{v}],[\DD_{n}^{v}]}, \\
    [\OO(2)^{-}]\circledcirc [\octa^{-}]     & = \set{[\triv],[\ZZ_{2}^{-}],[\DD_{3}^{v}],[\DD_{2}^{v}]},                     \\
    [\OO(2)^{-}]\circledcirc [\OO(2)^{-}]    & =\set{[\ZZ_{2}^{-}],[\OO(2)^{-}]}.
  \end{align*}
\end{lem}

\begin{proof}[Sketch of proof]
  We will only focus on the clips operation $[\OO(2)^{-}]\circledcirc [\DD_{2n}^h]$ and consider thus intersections
  \begin{equation*}
    \OO(2)^{-}\cap \left(g\DD_{2n}^hg^{-1}\right),\quad g\in \OO(3).
  \end{equation*}
  There are only two non trivial cases to work on, whether characteristic axes match or not.

  \begin{enumerate}
    \item If primary axes match, then, by~\eqref{eq:diedralh} and \eqref{eq:DecO2moins}, we have to consider intersections
          \begin{equation}\label{eq:IntersO2moinsD2nh}
            \ZZ^{\sigma_v}\cap \ZZ_{2n}^{-},\quad \SO(2)\cap \ZZ_{2}^{p_{j}},\quad \ZZ^{\sigma_{b}}\cap \ZZ_{2}^{\sigma_{q_{j}}}
          \end{equation}
          which reduce to $\DD_{n}^{v}$ (see decomposition \eqref{eq:diedralv}).

    \item Suppose, moreover, that $p_{0}=(Oz)$, in which case $\SO(2)\cap \ZZ_{2}^{p_{j}}=\ZZ_{2}$.
          \begin{enumerate}[(a)]
            \item For $n$ odd, there exists some secondary axes $q_k$ in the $xy$ plane (see Lemma~\ref{lem:Perp_Axe_Diedh}) and thus $\ZZ^{\sigma_{b}}\cap \ZZ_{2}^{\sigma_{q_{j}}}$ reduces to $\ZZ_{2}^{\sigma_{q_k}}$. Moreover, $\ZZ^{\sigma_v}\cap \ZZ_{2n}^{-}$ reduces to some $\ZZ_{2}^{\sigma_{v}}$ with $v$ perpendicular to $p_{0}$ and $q_{j}$ and the final result is a subgroup conjugate to $\DD_{2}^{v}$.
            \item For $n$ even, we obtain $\ZZ_{2}$.
          \end{enumerate}

    \item If now the primary axis $a$ of $g\DD_{2n}^hg^{-1}$ is $(Ox)$, and no other characteristic axes correspond to $(Oz)$ nor $(Oy)$, then intersections \eqref{eq:IntersO2moinsD2nh} reduce to $\ZZ_{2}^{\sigma_a}\cap \left(g\ZZ_{2n}^{-}g^{-1}\right)$ which is conjugate to $\ZZ_{2}^-$.
  \end{enumerate}
\end{proof}

\bibliographystyle{abbrv}
\bibliography{refs}

\end{document}